\documentclass{article}

\usepackage{hyperref}
\usepackage{graphicx}%
\usepackage{multirow}%
\usepackage{amsmath,amssymb,amsfonts}%
\usepackage{amsthm}%
\usepackage{mathrsfs}%
\usepackage[title]{appendix}%
\usepackage{xcolor}%
\usepackage{textcomp}%
\usepackage{manyfoot}%
\usepackage{booktabs}%
\usepackage{algorithm}%
\usepackage{algorithmicx}%
\usepackage{algpseudocode}%
\usepackage{listings}%
\usepackage{enumitem}%
\usepackage{cleveref}%
\usepackage{bbm,bm}%
\usepackage{tikz}%

\theoremstyle{plain}%
\newtheorem{theorem}{Theorem}%
\newtheorem{proposition}{Proposition}%
\newtheorem{lemma}{Lemma}%

\theoremstyle{definition}%
\newtheorem{definition}{Definition}%

\theoremstyle{remark}%
\newtheorem{remark}{Remark}%

\newcommand{\bC}{\mathbb{C}}
\newcommand{\bH}{\mathbb{H}}
\newcommand{\bN}{\mathbb{N}}
\newcommand{\bR}{\mathbb{R}}
\newcommand{\cA}{\mathcal{A}}
\newcommand{\cF}{\mathcal{F}}
\newcommand{\cI}{\mathcal{I}}
\newcommand{\cK}{\mathcal{K}}
\newcommand{\cO}{\mathcal{O}}
\newcommand{\cS}{\mathcal{S}}
\newcommand{\cT}{\mathcal{T}}
\newcommand{\rd}{\mathrm{d}}
\newcommand{\rc}{\mathrm{c}}

\newcommand{\diag}{\mathrm{diag}}
\newcommand{\eps}{\varepsilon}
\newcommand{\norm}[1]{\|{#1}\|}
\newcommand{\nnorm}[1]{\left\|{#1}\right\|}
\newcommand{\od}[2]{\dfrac{\rd{#1}}{\rd{#2}}}
\newcommand{\pd}[2]{\dfrac{\partial{#1}}{\partial{#2}}}

\title{Lax Equivalence for Hyperbolic Relaxation Approximations}
\author{Zeyu Jin \footnote{School of Mathematical Sciences, Peking University, Beijing 100871, China (\texttt{jinzy@pku.edu.cn}).} 
\and Ruo Li
\footnote{CAPT, LMAM and School of Mathematical Sciences, Peking University, Beijing 100871, China; Chongqing Research Institute of Big Data, Peking University, Chongqing 401121, China
(\texttt{rli@math.pku.edu.cn}).}}
\date{}

\begin{document}
\maketitle

\begin{abstract}
  This paper investigates the zero relaxation limit for general 
  linear hyperbolic relaxation systems and establishes the 
  asymptotic convergence of slow variables under the unimprovable 
  weakest stability condition, akin to the Lax equivalence theorem 
  for hyperbolic relaxation approximations.
  Despite potential high oscillations, the convergence of macroscopic 
  variables is established in the strong $L^\infty_t L^2_x$ sense 
  rather than the sense of weak convergence, time averaging, or 
  ensemble averaging.
  \newline

  \newline

\end{abstract}

\section{Introduction}
This work is concerned with the Cauchy problems for first-order 
linear hyperbolic relaxation systems with stiff source terms in 
several spatial variables:
\begin{equation}\label{eq}
  \pd{U}{t} + \sum_{j = 1}^d A_j \pd{U}{x_j} = \frac{1}{\eps} Q U, 
  \quad U(0, x, \eps^{-1}) = U_0(x),
\end{equation}
where 
the relaxation time $\eps > 0$ is a small parameter, 
$U$ is the unknown $n$-vector function of 
$(t, x, \eta) = (t, x_1, x_2, \ldots, x_d, \eta) \in [0, +\infty) 
\times \bR^d \times [0, +\infty)$ with $\eta = \eps^{-1}$, and 
the matrices $Q \in \bR^{n \times n}$ and $A_j \in \bR^{n \times n}$ 
with $j = 1, 2, \ldots, d$ are constant matrices.
\Cref{eq} serves as a linearized model problem for general hyperbolic 
relaxation systems arising in fields such as
kinetic theory \cite{Cercignani1994dilute,jin2023NMR,
Platkowski1988DVM}, 
inviscid gas dynamics \cite{Zeng1999gas}, and 
non-equilibrium thermodynamics \cite{Chapman1939nonuniform,
Muller1993extended,Muller1998rational}.

This work investigates the behavior of solutions to 
\cref{eq} as the relaxation time $\eps$ approaches zero, 
referred to as the \textit{zero relaxation} or 
\textit{asymptotic limit}. 
It focuses on the slow time scale, commonly considered in 
many practical applications \cite{Kreiss1980problems}, 
rather than the fast time scale.
The primary concern is the convergence of slow 
variables in the zero relaxation limit under specific stability 
conditions of \cref{eq}.

Extensive research with a long history revolves around 
investigating the asymptotic behavior in the zero relaxation limit 
for hyperbolic relaxation systems. 
Interested readers seeking a more thorough exploration 
of this problem can refer to 
\cite{Natalini1999recent,Yong_singular_1999,Yong_basic_2002} and the 
references cited therein. 
Many works in this field trace their origins back to 
T.-P. Liu's seminal work \cite{Liu1987Hyperbolic} on general 
one-dimensional $2 \times 2$ systems. 
W.-A. Yong established a general theory for first-order quasi-linear 
symmetrizable hyperbolic systems 
\cite{Yong_singular_1992,Yong_singular_1999,Yong2001basic,
Yong_basic_2002}. 
Yong's structural stability conditions generalize the 
subcharacteristic conditions for one-dimensional $2 \times 2$ systems 
\cite{Liu1987Hyperbolic,Whitham1974linear} and the 
time-like conditions for one-dimensional scalar second-order 
hyperbolic systems \cite{deJager1975singular,Geel1978hyperbolic,
Schochet1987hyperbolic,vanHarten1985hyperbolic}. 
Several studies also assume that the systems satisfy specific 
entropy conditions \cite{Boillat1997hyperbolic,Chen1994hyperbolic,
Kawashima2004dissipative,Ruggeri2004stability,Tzavaras2005relative,
Yong2004entropy}. 
The investigation of convergence rates in asymptotic convergence can 
be found in \cite{Kurganov1997stiff,Teng1998first,Tveito1997rate},
and the design of relaxation schemes based on these  
theories is discussed in 
\cite{JinXin1995relaxation,Natalini1996Convergence}. 
The linear systems under examination, as formulated in \cref{eq}, 
have been studied by J. Lorenz and H. J. Schroll 
\cite{lorenz_stiff_1997,Lorenz1999Hyperbolic}, 
who established the equivalence in some sense between the 
asymptotic convergence of the solutions and the condition of 
stiff well-posedness. 
All these works mentioned above are based on the non-oscillation 
assumption, a simple case elucidated in \cite{Yong_basic_2002},
since dissipation provides a coercive term for energy estimates
\cite{Chern1995hyperbolic,Yong_singular_1999}. 
The literature for systems exhibiting oscillatory behavior is less  
extensive than in the dissipative case.
Such oscillatory systems arise in many problems in 
fluid mechanics and plasma physics, for instance, in the study of 
highly rotating fluids 
\cite{Babin1996global,Embid1996averaging,Fanelli2019Asymtptocs,
Feireisl2012Multiscale}, 
the quasi-neutral limit of plasma \cite{Grenier1996Oscillations}, 
and molecular dynamics 
\cite{Ascher1999oscillatory,Ebert1981Modelling}.
H.-O. Kreiss \cite{Kreiss1980problems} established uniform estimates 
in $\eps$ for these systems, assuming that each $A_j$ is Hermitian 
and $Q$ is skew-Hermitian. 
Pseudo-differential energy estimates for hyperbolic systems with 
singular perturbations are provided in \cite{Grenier1997Pseudo}.

Through a review of the historical trajectory of this problem, one 
can observe that one of the primary concerns 
is to identify a general set of structural properties or axioms for 
the systems and then develop the corresponding mathematical theories 
\cite{Yong2008interesting}. 
A natural question arises with a focus on the slow time scale: 
Under what weak structural conditions can slow variables still 
exhibit asymptotic convergence? 
Furthermore, given the development of hyperbolic relaxation 
approximations \cite{JinXin1995relaxation,Natalini1996Convergence} 
and their similarity to numerical approximations, another crucial 
question is whether it also has a corresponding Lax equivalence 
theorem \cite{Lax1956Survey}.
These questions form the central concern of the present 
research endeavor.

Specifically, we aim to eliminate the unnecessary non-oscillation 
assumption. 
As highlighted in \cite{Yong_basic_2002}, the weakest requirement for 
\cref{eq} is the condition of stiff well-posedness, equivalent to 
the $L^2$ boundedness of solutions in the asymptotic limit 
for arbitrary $L^2$ initial values.
At this point, the dynamics of fast variables should allow both 
fast dissipation and oscillation. 
However, a difficulty arises: energy estimates, as used in 
\cite{Yong_singular_1999}, cannot be applied due to the lack of 
dissipation. 
By employing the Fourier transform and solving the corresponding 
ordinary differential equations, the solutions can be 
expressed as a matrix exponential with a small parameter $\eps$. 
Nevertheless, directly studying its asymptotic behavior remains 
challenging.

This paper establishes the asymptotic convergence of slow 
variables in the zero relaxation limit for \cref{eq} under 
the weakest stability condition, namely, stiff well-posedness, 
which is notably weaker than the previous stability conditions.
To address the challenge posed by rapid oscillation in fast 
variables, we employ the Mori--Zwanzig technique 
\cite{E2011Principles,mori1965transport,zwanzig1960collision} 
to decouple fast and 
slow variables and to focus exclusively on the latter. 
The pivotal term involving the convolution of two matrix exponentials 
is estimated at the low-frequency region using a generalized 
Riemann--Lebesgue lemma for matrix exponentials.
By utilizing the Laplace transform, the proof of this lemma is 
reduced to estimating an integral of resolvents. 
A sharper version of resolvent estimates is applied to overcome 
limitations in the canonical Kreiss matrix theorem. 
To tackle the difficulty in estimating high-frequency 
error under the weakest condition, i.e., stiff well-posedness, we 
first establish the strong hyperbolicity of reduced systems by 
constructing 
a Picard sequence and using the Banach fixed-point theorem, and then 
indirectly bound the error at the high-frequency region.
Combining estimates at high- and low-frequency regions yields the 
desired convergence and the corresponding error estimate.

The rest of this paper is organized as follows. 
\Cref{sec:preliminaries} introduces several fundamental preliminaries 
that form the basis of this work and presents our main result. 
\Cref{sec:proof} demonstrates the proof of the asymptotic 
convergence and the corresponding error estimate of slow variables 
under stiff well-posedness.
\Cref{sec:concl} provides a summary and some concluding remarks. 
\Cref{app:kreiss} discusses the improved resolvent estimate of the 
Kreiss matrix theorem, which is essential in the proof of our main 
result.

\section{Preliminaries}
\label{sec:preliminaries}
This section introduces two fundamental concepts: the Kreiss matrix 
theorem and the stability conditions for \cref{eq}, with a particular 
emphasis on stiff well-posedness.
These concepts form the foundation of this study. In addition, 
the main result of this paper is presented.

\subsection{Kreiss matrix theorem}
The Kreiss matrix theorem 
\cite{kreiss_uber_1959,kreiss_initial-boundary_2004} 
is one of the most fundamental results of 
the well-posedness for Cauchy problems in the theory of partial 
differential equations.
For the sake of simplicity, let us introduce the following 
definitions. 
Similar definitions can be found in \cite{Yong_singular_1992}.
\begin{definition}[Quasi-stability and uniform quasi-stability]
  A matrix $M \in \bC^{n \times n}$ is called \textit{quasi-stable} 
  if 
  \begin{displaymath}
    \sup_{t \ge 0} \norm{e^{M t}} < + \infty.
  \end{displaymath}
  A set of matrices $\cF \subset \bC^{n \times n}$ is 
  called \textit{uniformly quasi-stable} if 
  \begin{displaymath}
    \sup_{M \in \cF} \sup_{t \ge 0} \norm{e^{M t}} < + \infty.
  \end{displaymath}
\end{definition}

Let us first recall the equivalent characterization of 
the quasi-stability of a single complex matrix in 
\cite{kreiss_initial-boundary_2004}.
The following lemma shows that a matrix is quasi-stable if and only 
if all its eigenvalues have real parts no greater than zero and all 
eigenvalues whose real parts are equal to zero are semi-simple.

\begin{lemma}[See Lemma 2.3.1 of \cite{kreiss_initial-boundary_2004}]
  \label{lemma:QS_character}
  For any $M \in \bC^{n \times n}$, the following two conditions are 
  equivalent.
  \begin{enumerate}
    \item The matrix $M$ is quasi-stable. 
    \item All eigenvalues $\lambda$ of the matrix $M$ have a real 
      part $\Re(\lambda) \le 0$. Furthermore, if $J_r$ is a Jordan 
      block of the Jordan matrix $J = S M S^{-1}$, corresponding 
      to an eigenvalue $\lambda_r$ of the matrix $M$ with 
      $\Re(\lambda_r) = 0$, then $J_r$ has dimension $1 \times 1$.
  \end{enumerate}
\end{lemma}

The Kreiss matrix theorem gives several 
necessary and sufficient conditions for the uniform quasi-stability 
of a set of matrices. The theorem is stated as follows.
\begin{theorem}[Kreiss matrix theorem. See Theorem 2.3.2 of 
  \cite{kreiss_initial-boundary_2004}]
  \label{thm:KMT}
  Let $\cF$ denote a set of matrices in $\bC^{n \times n}$.
  The following four conditions are equivalent.
  \begin{enumerate}
    \item There exists a constant $K_1$ such that  
      $\norm{e^{M t}} \le K_1$ for each $M \in \cF$ and 
      $t \ge 0$.
    \item There exists a constant $K_2$ such that 
      \begin{displaymath}
        \norm{(z I - M)^{-1}} \le \frac{K_2}{\Re(z)}, 
      \end{displaymath}
      for each $M \in \cF$ and $\Re(z) > 0$.
    \item There exist constants $K_{31}$, $K_{32}$ such that 
      for each $M \in \cF$, there exists a 
      transformation $S = S(M)$ with 
      $\norm{S} + \norm{S^{-1}} \le K_{31}$,
      the matrix $S M S^{-1}$ is upper triangular,
      \begin{displaymath}
        S M S^{-1} = \begin{pmatrix}
          b_{11} & b_{12} & \cdots & b_{1n} \\
               & b_{22} & \cdots & b_{2n} \\
               &    & \ddots & \vdots \\
               &    &    & b_{nn} \\
        \end{pmatrix},
      \end{displaymath}
      the diagonal is ordered, 
      \begin{displaymath}
        0 \ge \Re(b_{11}) \ge \Re(b_{22}) \ge \cdots \ge 
        \Re(b_{nn}),
      \end{displaymath}
      and the upper diagonal elements satisfy the estimate 
      \begin{displaymath}
        |b_{ij}| \le K_{32} |\Re(b_{ii})|, \quad 
          1 \le i < j \le n.
      \end{displaymath}
    \item There exists a positive constant $K_4$ such that 
      for each $M \in \cF$, there exists a Hermitian 
      matrix $A_0 = A_0(M)$ such that 
      \begin{displaymath}
        K_4^{-1} I \le A_0 \le K_4 I, \quad A_0 M + M^* A_0 \le 0.
      \end{displaymath}
  \end{enumerate}
\end{theorem}

Among other things, the theorem establishes the equivalence of the 
uniform boundedness of semigroups and certain resolvent estimates 
for a set of matrices. 
However, Kreiss's resolvent condition sometimes fails to provide  
sharp estimates for well-posed Cauchy problems.
To elucidate an essential technique used in our proof and the 
motivation behind improving the resolvent estimates in the Kreiss 
matrix theorem, let us consider the following Cauchy problem:
\begin{equation}\label{eq:cauchy}
  \pd{u}{t} = \cA u + f(t, x), \quad u(0, x) = 0,
\end{equation}
where $u$ is the unknown $n$-vector function of 
$(t, x) = (t, x_1, x_2, \ldots, x_d) \in [0, +\infty) \times \bR^d$,
the source term $f$ is an $n$-vector function of $(t, x)$,
the differential operator $\cA$ is defined by 
$(\cA v)^{\wedge}(\xi) := A(\xi) \hat{v}(\xi)$, where 
$\hat{v}$ is the Fourier transform of an $n$-vector function $v$ 
defined on $\bR^d$, 
and the symbol of $\cA$, denoted as $A(\xi)$, is an 
$n \times n$-matrix function of $\xi \in \bR^d$.
Under the assumption of well-posedness 
\cite{kreiss_initial-boundary_2004} of \cref{eq:cauchy}, the Kreiss 
matrix theorem implies the existence of constants $\alpha \in \bR$ 
and $K > 0$ independent of $\xi$ such that the following resolvent 
estimate holds:
\begin{equation}\label{eq:cauchy_KMTbound}
  \|{(z I - A(\xi))^{-1}}\| \le \frac{K}{\Re(z) - \alpha}, 
\end{equation}
for each $z \in \bC$ whose real part satisfies $\Re(z) > \alpha$.
By applying the Fourier transform and the Laplace transform to both 
sides of \cref{eq:cauchy}, one can derive the following equation:
\begin{displaymath}
  z \tilde{u}(z, \xi) = A(\xi) \tilde{u}(z, \xi) 
    + \tilde{f}(z, \xi),
\end{displaymath}
under some appropriate assumptions on the functions $A(\xi)$ and 
$f(t, x)$. Here, $\tilde{u}(z, \xi)$ is the Laplace transform with 
respect to the variable $t$ of the Fourier transform $\hat{u}(t, \xi)$
of the function $u(t, x)$.
Therefore, it follows that $\tilde{u}(z, \xi) = (z I - A(\xi))^{-1} 
\tilde{f}(z, \xi)$.
The inverse Laplace transform of $\tilde{u}(z, \xi)$ is determined 
by  
\begin{displaymath}
  \hat{u}(t, \xi) = \frac{1}{2 \pi i} 
    \int_{\gamma - i \infty}^{\gamma + i \infty}
    e^{z t} (z I - A(\xi))^{-1} \tilde{f}(z, \xi) \,\rd z,
\end{displaymath}
where $\gamma$ is a real number such that the contour line lies in 
the region of convergence of $\tilde{u}(z, \xi)$ for each 
$\xi \in \bR^n$, and the integral is understood in the sense of 
the principal value.
One may estimate $\hat{u}(t, \xi)$ as 
\begin{displaymath}
  \|{\hat{u}(t, \xi)}\| \le \frac{e^{\gamma t}}{2 \pi} 
    \int_{- \infty}^{+ \infty} \|{(z I - A(\xi))^{-1}}\|
    \cdot \|{\tilde{f}(z, \xi)}\| \,\rd y,
\end{displaymath}
where $z = \gamma + i y$. The Kreiss matrix theorem provides the 
resolvents with an upper bound, as given in 
\cref{eq:cauchy_KMTbound}, which depends only on the real part 
of $z$. However, this bound is too rough, as for fixed $\xi$ and 
$\gamma$, it holds that $\|{(z I - A(\xi))^{-1}}\| = \cO(|y|^{-1})$
as $|y| \rightarrow + \infty$,
while \cref{eq:cauchy_KMTbound} cannot capture the decay of the 
resolvents.
To obtain a more accurate upper bound for $\hat{u}(t, \xi)$, we 
require sharper resolvent estimates.

Henceforth, we denote the open right-half complex plane as 
\begin{displaymath}
  \bH = \left\{ z \in \bC \,\big|\, \Re(z) > 0 \right\}.
\end{displaymath}
For a matrix $M \in \bC^{n \times n}$, the spectrum of $M$ is 
denoted as $\sigma(M)$, and we define a new measurement of uniform 
quasi-stability as:
\begin{displaymath}
  \cK(M) := \sup_{z \in \bH} \frac{\norm{(z I - M)^{-1}}}
  {\max_{\lambda \in \sigma(M) \setminus \bH} |z - \lambda|^{-1}}.
\end{displaymath}
When $\sigma(M)$ is a subset of the half-plane $\bH$, we set 
$\cK(M) := + \infty$.

The improved resolvent estimates in the Kreiss matrix theorem are 
stated as follows.

\begin{theorem}\label{thm:IKMT}
  Let $\cF$ be a set of matrices in $\bC^{n \times n}$.
  The matrix set $\cF$ is uniformly quasi-stable if and only if 
  \begin{equation}\label{eq:uniformKM}
    \sup_{M \in \cF} \cK(M) < + \infty.
  \end{equation}
  In this case, there exists a positive constant $K$ such that 
  \begin{equation}\label{eq:main_resolvent}
    \norm{(z I - M)^{-1}} \le K \max_{\lambda \in \sigma(M)} 
    |z - \lambda|^{-1},
  \end{equation}
  for each $M \in \cF$ and $z \in \bH$.
\end{theorem}

The proof of this theorem will be provided in \Cref{app:kreiss}.

\begin{remark}
  Numerous studies have been dedicated to improving the uniform upper 
  bounds of the semigroups under Kreiss's resolvent condition, 
  as investigated in \cite{leveque_resolvent_1984,van_linear_1993}. 
  The focus of \Cref{thm:IKMT} is on the converse, i.e., improving 
  the resolvent estimates for a uniformly quasi-stable matrix set.
  Several attempts have been made in this direction. For instance,  
  J. Miller \cite{miller_resolvent_1968} estimated the resolvents 
  along specific contours in the left-half complex plane by 
  categorizing the spectrum of each matrix. 
  R. Zarouf \cite{zarouf_sharpening_2009} refined 
  the resolvent estimates for power-bounded matrices using 
  Bernstein-type inequalities for rational functions.
  Detailed discussions of these results and their relations to 
  \Cref{thm:IKMT} are presented in \Cref{app:kreiss}.
\end{remark}

\subsection{Stability conditions}

Let us now review several stability conditions for \cref{eq} in 
the existing literature, with a specific focus on the concept of 
stiff well-posedness. 
We aim to elucidate the relations between these stability conditions 
and present the main result in this paper.

The first and second stability conditions proposed by W.-A. Yong 
\cite{Yong2001basic} are presented as follows.
Under these conditions, Yong developed a theory of singular 
perturbation for general first-order quasi-linear symmetrizable 
hyperbolic systems with stiff source terms 
\cite{Yong_singular_1992,Yong_singular_1999,
Yong2001basic,Yong_basic_2002}. 
A basic assumption of Yong's stability conditions is that zero is a 
semi-simple eigenvalue of the matrix $Q$, i.e., there exists an 
invertible $n \times n$ matrix $P$ and an invertible $r \times r$ 
matrix $B$ such that 
\begin{equation}\label{eq:Yong1_PQPinv}
  P Q P^{-1} = \begin{pmatrix}
    0 & 0 \\ 0 & B
  \end{pmatrix}.
\end{equation}

\begin{definition}[Yong's stability conditions \cite{Yong2001basic}]
  We say that \cref{eq} satisfies 
  \textit{Yong's first stability condition} 
  if the following requirements hold:
  \begin{enumerate}[label={(\roman{enumi})}]
    \item \label{cond:Yong1}
    All the eigenvalues of the matrix $B$ in \cref{eq:Yong1_PQPinv} 
    have negative real parts;
    \item \label{cond:Yong2}
    \Cref{eq} is \textit{symmetrizable hyperbolic}, i.e., there 
    exists a positive definite Hermitian matrix $A_0$ such that 
    $A_0 A_j = A_j^* A_0$
    for each $j = 1, 2, \ldots, d$;
    \item \label{cond:Yong3w}
    The hyperbolic and source terms satisfies 
    \begin{displaymath}
      A_0 Q + Q^* A_0 \le 0.
    \end{displaymath}
  \end{enumerate}
  We say that \cref{eq} satisfies 
  \textit{Yong's second stability condition} 
  if it satisfies the requirements 
  \ref{cond:Yong1}, \ref{cond:Yong2}, and 
  \begin{enumerate}[start=3,label={(\roman{enumi})$^\prime$}]
    \item \label{cond:Yong3s}
    The hyperbolic and source terms are coupled in the sense that 
    \begin{displaymath}
      A_0 Q + Q^* A_0 \le - P^* \begin{pmatrix}
        0 & 0 \\ 0 & I_r
      \end{pmatrix} P.
    \end{displaymath}
  \end{enumerate}
\end{definition}

\begin{remark}[Non-oscillation assumption]\label{rmk:yong->noa}
  \Cref{cond:Yong1} is equivalent to the statement that the matrix 
  $Q$ is quasi-stable and has no purely imaginary eigenvalues 
  different from zero, which is just the 
  \textit{non-oscillation assumption} \cite{Yong_basic_2002}.
  Under the basic assumption presented in \cref{eq:Yong1_PQPinv}, 
  \Cref{cond:Yong3s} implies \Cref{cond:Yong1} as shown in 
  \cite[Lemma 2.1]{Yong_singular_1999}.
\end{remark}

\begin{remark}\label{rmk:relation_yong}
  Symmetrizable hyperbolicity is an easily verifiable condition for 
  the following first-order partial differential equations with 
  non-stiff source terms:
  \begin{equation}\label{eq:1st}
    \pd{U}{t} + \sum_{j = 1}^d A_j \pd{U}{x_j} = Q U.
  \end{equation}
  However, it is not the weakest condition to ensure the 
  well-posedness of \cref{eq:1st} as pointed out in 
  \cite{kreiss_initial-boundary_2004}.
  \Cref{cond:Yong2,cond:Yong3w} in Yong's first stability condition 
  can be seen as a natural extension of symmetrizable hyperbolicity 
  in the case of hyperbolic relaxation systems presented in \cref{eq}.
\end{remark}

Let us now introduce the definition of stiff well-posedness,
which forms the foundation of this work.

\begin{definition}[Stiff well-posedness]
  \label{def:stiff_well_posedness}
  We say that \cref{eq} satisfies the condition of 
  \textit{stiff well-posedness} or it is \textit{stiffly well-posed}
  if the matrix set, 
  \begin{displaymath}
    \cF_0 := \left\{ H(\xi, \eta) \,\big|\, \xi \in \bR^d, \eta \ge 0 
    \right\},
  \end{displaymath}
  is uniformly quasi-stable, where 
  \begin{equation}\label{eq:H}
    H = H(\xi, \eta) := \eta Q - i \sum_{j = 1}^d \xi_j A_j.
  \end{equation}
\end{definition}

\begin{remark}[Strong hyperbolicity]
  Parallel to the discussions in \Cref{rmk:relation_yong}, stiff 
  well-posedness can be regarded as a natural extension of 
  strong hyperbolicity \cite{kreiss_initial-boundary_2004} 
  within the context of hyperbolic relaxation systems.
  We say that \cref{eq:1st} is \textit{strongly hyperbolic} if the 
  matrix set, 
  \begin{displaymath}
    \cF_{00} := \left\{ H(\xi, 0) \,\big|\, \xi \in \bR^d \right\},
  \end{displaymath}
  is uniformly quasi-stable, where $H(\xi, \eta)$ is defined by 
  \cref{eq:H}.
  According to \cite{kreiss_initial-boundary_2004}, strong 
  hyperbolicity is equivalent to the well-posedness of \cref{eq:1st}.
\end{remark}

The condition of stiff well-posedness is the most minimal  
stability requirement to ensure the well-posedness of \cref{eq} in 
the asymptotic limit. it is equivalent to 
the $L^2$ boundedness of solutions in the asymptotic limit for 
arbitrary $L^2$ initial values in some sense, as shown by the 
following proposition.

\begin{proposition}[See Theorem 2.2 of \cite{Yong_basic_2002}]
  \label{prop:not_SWP}
  Assume that \cref{eq} is strongly hyperbolic. If 
  \cref{eq} is not stiffly well-posed, then for any $t > 0$, there 
  exists $U_0 \in L^2$ such that 
  \begin{displaymath}
    \mathop{\lim \sup}_{\eps \to 0} 
    \norm{U(t, \cdot, \eps^{-1})}_{L^2} = + \infty.
  \end{displaymath}
\end{proposition}

\begin{remark}[Slow and fast variables]
  Under the condition of stiff well-posedness, the matrix $Q$ is 
  quasi-stable, and thus, zero is a semi-simple eigenvalue of $Q$, 
  as indicated in the basic assumption in \cref{eq:Yong1_PQPinv}.
  At this point, one can consider the following equation concerning 
  the unknown function $P U$:
  \begin{displaymath}
    \pd{(P U)}{t} + \sum_{j = 1}^d P A_j P^{-1} \cdot \pd{(P U)}{x_j} 
    = \frac{1}{\eps} P Q P^{-1} \cdot P U,
  \end{displaymath}
  and the symmetrizer $A_0$ can be substituted by $P A_0 P^{-1}$.
  Therefore, henceforth, when assuming the quasi-stability of 
  the matrix $Q$, we consider the form of $Q$ to be 
  \begin{displaymath}
    Q = \begin{pmatrix} 0 & 0 \\ 0 & B \end{pmatrix},
  \end{displaymath}
  without loss of generality. 
  Correspondingly, the matrices $A_j$ and $H$, and the vector $U$, 
  can be expressed in block form as follows:
  \begin{displaymath}
    A_j = \begin{pmatrix}
      A_{j,11} & A_{j,12} \\ A_{j,21} & A_{j,22}
    \end{pmatrix}, \quad 
    H = \begin{pmatrix}
      H_{11} & H_{12} \\ H_{21} & H_{22}
    \end{pmatrix}, \quad 
    U = \begin{pmatrix}
      u \\ v
    \end{pmatrix},
  \end{displaymath}
  where $u$ and $v$ represent \textit{slow variables} and 
  \textit{fast variables}, respectively.
\end{remark}

\begin{remark}\label{rmk:yong->SWP}
  The relation between stiff well-posedness and
  \Cref{cond:Yong2,cond:Yong3w} in Yong's first stability condition, 
  as well as the relation between strong hyperbolicity and 
  symmetrizable hyperbolicity, can be elucidated from the perspective 
  of the Kreiss matrix theorem presented in \Cref{thm:KMT}. 
  Let us consider the former as an example.
  According to \Cref{thm:KMT}, stiff well-posedness is equivalent to  
  the existence of a positive constant $K_4$ such that for each 
  $\xi \in \bR^d$ and $\eta \ge 0$, there exists a Hermitian matrix 
  $A_0 = A_0(\xi, \eta)$ such that 
  \begin{equation}\label{eq:SWP_SPD}
    K_4^{-1} I \le A_0(\xi, \eta) \le K_4 I, \quad 
    A_0(\xi, \eta) H(\xi, \eta) + H(\xi, \eta)^* A_0(\xi, \eta) \le 0.
  \end{equation}
  From \Cref{cond:Yong2,cond:Yong3w}, one can deduce the existence of 
  a positive definite Hermitian matrix $A_0$ such that 
  \begin{equation}\label{eq:yong_SPD}
    A_0 H(\xi, \eta) + H(\xi, \eta)^* A_0 \le 0,
  \end{equation}
  for each $\xi \in \bR^d$ and $\eta \ge 0$.
  Conversely, if \cref{eq:yong_SPD} holds for each $\xi \in \bR^d$ 
  and $\eta \ge 0$, \Cref{cond:Yong2,cond:Yong3w} also hold by taking 
  $(\xi, \eta) = (0, 1), (\pm e_j, 0)$ with $j = 1, 2, \ldots, d$.
  Therefore, Yong's first stability condition without the 
  non-oscillation assumption corresponds to the case when the 
  Hermitian matrix $A_0(\xi, \eta)$ in \cref{eq:SWP_SPD} is 
  independent of $\xi$ and $\eta$. 
  In this case, it can be shown that $A_0$ is a block diagonal matrix 
  \cite[Lemma 2.1]{Yong_singular_1999}, leading effortlessly to the 
  uniform quasi-stability of the following two matrix sets:
  \begin{gather}
    \label{eq:F1}
    \cF_1 := \left\{ H_{11}(\xi, 0) \,\big|\, \xi \in \bR^d \right\},
    \\
    \label{eq:F2}
    \cF_2 := \left\{ H_{22}(\xi, \eta) \,\big|\, 
    \xi \in \bR^d, \eta \ge 0 \right\},
  \end{gather}
  where the former aligns with the conclusion in 
  \Cref{lemma:UQS_of_F1},
  ensuring the strong hyperbolicity of reduced systems presented in 
  \cref{eq:reduced}, and the latter facilitates the estimation of
  a crucial term defined in \cref{eq:key_term} in the proof of our 
  main result. 
  However, it is important to note that 
  the only assumption for \cref{eq} in this paper to 
  establish asymptotic convergence of slow variables is stiff 
  well-posedness, demanding additional efforts for the proof.
\end{remark}

We have shown in \Cref{rmk:yong->noa,rmk:yong->SWP}
that Yong's first stability condition implies the 
\textit{Lorenz--Schroll condition} \cite{lorenz_stiff_1997}, which 
comprises the non-oscillation assumption and the stiff 
well-posedness condition. J. Lorenz and H. J. Schroll established 
the equivalence, in some sense, between the convergence of the 
solutions in the zero relaxation limit and the Lorenz--Schroll 
condition \cite{lorenz_stiff_1997}.

There are two essentially different asymptotic behaviors of 
the solutions to \cref{eq}, depending 
on whether $B$ has non-zero purely imaginary eigenvalues,
corresponding to rapid oscillation and dissipation of fast 
variables.
The works mentioned earlier in this section considered 
the simpler case when the non-oscillation assumption holds, 
as pointed out in \cite{Yong_basic_2002}. 
This paper aims to eliminate the non-oscillation 
assumption and investigate the asymptotic behavior of \cref{eq} 
subject to the only condition of stiff well-posedness, 
the weakest requirement that can be posed on \cref{eq}.
Similar to the results in 
\cite{lorenz_stiff_1997,Yong_singular_1999}, 
this work aims to prove the asymptotic convergence of slow 
variables toward the solution to the following reduced system, 
\begin{equation}\label{eq:reduced}
  \pd{\breve{u}}{t} + \sum_{j = 1}^{d} A_{j,11} \pd{\breve{u}}{x_j} = 
  0, \quad \breve{u}(0, x) = u_0(x).
\end{equation}

The main result of this paper is presented as follows.

\begin{theorem}\label{thm}
  If \cref{eq} is stiffly well-posed, then the reduced system 
  \eqref{eq:reduced} is strongly hyperbolic, and the following 
  convergence holds for each $t > 0$ and $U_0 \in L^2$:
  \begin{displaymath}
    \lim \limits_{\eps \to 0} 
    \| u(t, \cdot, \eps^{-1}) - \breve{u}(t, \cdot) 
    \|_{L^2} = 0,
  \end{displaymath}
  where $\breve{u}$ is determined by \cref{eq:reduced}.
  Furthermore, if the matrix set $\cF_2$ defined in \cref{eq:F2}
  is uniformly quasi-stable,
  then the following error estimate holds for each $t > 0$, 
  $U_0 \in H^2$ and sufficiently small $\eps$:
  \begin{displaymath}
    \| u(t, \cdot, \eps^{-1}) - \breve{u}(t, \cdot) 
    \|_{L^2} 
    \le C e^t \eps |\log \eps| \cdot \|U_0\|_{H^2},
  \end{displaymath}
  where $C$ is a constant dependent on the matrices $A^j$ and $Q$.
\end{theorem}

\begin{remark}
  As indicated by \Cref{prop:not_SWP}, stiff well-posedness is 
  a necessary stability criterion for \cref{eq} to have a correctly 
  behaved asymptotic limit \cite{Yong_basic_2002}.
  \Cref{thm} shows that this necessary criterion is sufficient to 
  imply the asymptotic convergence of slow variables.
\end{remark}

\begin{remark}
  As mentioned previously, dealing with high oscillations constitutes 
  a challenging aspect of the problem. 
  For multiscale problems with oscillatory or stochastic terms, one 
  would typically expect convergence in the sense of weak 
  convergence \cite{Fanelli2019Asymtptocs}, time averaging, or 
  ensemble averaging \cite{Pavliotis2008Multiscale}. 
  However, as demonstrated by the proof of \Cref{thm}, our findings 
  reveal uniform convergence within a finite time $t$, thereby 
  confirming the asymptotic convergence of slow variables over a 
  finite time interval in the strong $L^\infty_t L^2_x$ sense under 
  the unimprovable weakest stability condition. 
  Notably, the solution to the reduced system 
  \eqref{eq:reduced} is free of oscillations. 
  Consequently, our results indicate that the oscillations in the 
  slow variables, both in time and in space, are suppressed.
\end{remark}

\begin{remark}
  The observed strong convergence inspires investigating  
  systems characterized by high oscillation or fluctuation, as 
  exemplified in molecular dynamics simulations. 
  In contrast to the typically expected weak convergence, where 
  physical quantities are accurate only in a statistical or average 
  sense, strong convergence implies near-universal correctness 
  pointwise. 
  This insight enlightens the design and analysis of multiscale 
  methods for oscillatory systems, providing a novel perspective and 
  potential avenues for advancement.
\end{remark}

\begin{remark}
  Hyperbolic relaxation approximations 
  \cite{JinXin1995relaxation,Natalini1996Convergence}, in our 
  context, utilize hyperbolic relaxation systems \eqref{eq} 
  to approximate reduced systems \eqref{eq:reduced} by introducing a 
  small parameter $\eps$.
  This idea parallels finite difference methods in numerical 
  approximations, where convergence toward the solutions to 
  \cref{eq:reduced} is considered as the mesh size approaches 
  zero. In numerical analysis, the Lax equivalence theorem 
  \cite{Lax1956Survey} is fundamental, stating that a consistent 
  linear finite difference method for a well-posed linear initial 
  value problem is convergent if it is stable.
  As highlighted in \Cref{prop:not_SWP}, stiff well-posedness is the 
  most natural stability condition for hyperbolic relaxation 
  approximations. 
  From this perspective, \Cref{thm} establishes the Lax equivalence 
  theorem for hyperbolic relaxation approximations.
\end{remark}

\begin{remark}
  To obtain the error estimate in \Cref{thm}, we require uniform 
  quasi-stability of the matrix set $\cF_2$, which can be guaranteed 
  by Yong's first stability condition without the non-oscillation 
  assumption as discussed in \Cref{rmk:yong->SWP}.
\end{remark}

\section{Proof of \Cref{thm}}
\label{sec:proof}

In this section, we present the proof of our main result.

\subsection{Basic ideas in the proof}
To enhance the clarity and comprehension of the proof of \Cref{thm}, 
let us begin by introducing some basic ideas within the proof. 

One crucial technique used in the proof is the Mori--Zwanzig technique
\cite{E2011Principles,mori1965transport,zwanzig1960collision} 
to decouple fast and slow variables. 
This technique is a general strategy to 
integrate out a subset of variables in a problem.
For \cref{eq}, this strategy can be employed as follows.
Applying the Fourier transform on both sides of \eqref{eq} leads to 
the following system:
\begin{displaymath}
  \od{\hat{U}}{t} = H \hat{U}, \quad 
  \hat{U}(0, \xi, \eta) = \hat{U}_0(\xi),
\end{displaymath}
that is, 
\begin{gather*}
  \od{\hat{u}}{t} = H_{11} \hat{u} + H_{12} \hat{v}, \quad 
  \hat{u}(0, \xi, \eta) = \hat{u}_0(\xi), \\
  \od{\hat{v}}{t} = H_{21} \hat{u} + H_{22} \hat{v}, \quad
  \hat{v}(0, \xi, \eta) = \hat{v}_0(\xi).
\end{gather*}
Applying Duhamel's principle, one obtains that 
\begin{gather}
  \label{eq:duhamel_u}
  \hat{u}(t, \xi, \eta) = \exp(H_{11} t) \hat{u}_0(\xi) + \int_0^t 
  \exp(H_{11} (t - s)) H_{12} \hat{v}(s, \xi, \eta) \,\rd s, \\
  \label{eq:duhamel_v}
  \hat{v}(t, \xi, \eta) = \exp(H_{22} t) \hat{v}_0(\xi) + \int_0^t 
  \exp(H_{22} (t - s)) H_{21} \hat{u}(s, \xi, \eta) \,\rd s, 
\end{gather}
which yields that 
\begin{multline}\label{eq:MZ}
  \hat{u}(t, \xi, \eta) = \exp(H_{11} t) \hat{u}_0(\xi) \\ + 
  G(t, \xi, \eta) \hat{v}_0(\xi) + 
  \int_{0}^{t} G(t - s, \xi, \eta) H_{21} \hat{u}(s, \xi, \eta) 
  \,\rd s,
\end{multline}
where 
\begin{equation}\label{eq:key_term}
  G(t, \xi, \eta) := \int_{0}^{t} \exp ( H_{11} (t - s) )
  H_{12} \exp(H_{22} s) \,\rd s.
\end{equation}

Note that the first term on the right-hand side of \cref{eq:MZ} is 
the Fourier transform of the solution $\breve{u}$ to 
\cref{eq:reduced}.
Therefore, we aim to demonstrate the $L^2$ norm of the sum of the 
second and third terms tends to zero as $\eta$ tends to infinity.
The expressions of these terms highlight the importance of 
$G(t, \xi, \eta)$.
At first glance, it suffices to establish that $G(t, \xi, \eta)$ 
vanishes in some sense in the asymptotic limit.

Due to the stiff well-posedness of \cref{eq}, the matrix $B$ is 
quasi-stable.
Recall that the matrix $H_{22}$ is defined as $H_{22} = 
H_{22}(\xi, \eta) = \eta B - i \sum_{j = 1}^n \xi_j A_{j, 22}$.
The asymptotic behavior of $G(t, \xi, \eta)$ is somewhat analogous to 
the limit process in the Riemann--Lebesgue lemma. 
This observation involves replacing the scalar term that exhibits 
high-frequency 
oscillation in the canonical Riemann--Lebesgue lemma with a matrix 
exponential $\exp(H_{22}(\xi, \eta) t)$ that may exhibit 
high-frequency oscillation or rapid dissipation in the variable $t$ 
when $\eta$ is large.
However, handling the matrix exponential in this form is challenging 
due to the non-commutative nature of the matrices $B$ and $A_{j, 22}$.

By Laplace transform, we reduce this problem to resolvent estimates.
Utilizing the improved resolvent estimates in \Cref{thm:IKMT} and 
estimates of an elementary integral in \Cref{lemma:integral}, 
we derive the generalized Riemann--Lebesgue lemma for matrix 
exponentials in \Cref{lemma:gRL}.
According to this lemma, a uniform asymptotic convergence estimate of 
$G(t, \xi, \eta)$ is obtained in \Cref{lemma:low_frequency_G},
specifically at the low-frequency region, i.e.,
when the frequency $\xi$ is bounded.
At this time, the low-frequency estimate of the error is 
summarized in \Cref{lemma:low_frequency}.

When it comes to estimating the high-frequency error, the 
problem is more difficult. 
As mentioned in \Cref{rmk:yong->SWP}, if \cref{eq} satisfies Yong's 
first stability condition without the non-oscillation assumption, the 
two sets of matrices, $\cF_1$ and $\cF_2$, are uniformly  
quasi-stable, which yields at least that $G(t, \xi, \eta)$ is 
bounded by $|\xi|$ up to multiplication by a constant at the 
high-frequency region. 
At this time, the error at the high-frequency region can be 
controlled by taking a sufficiently large frequency truncation for 
the given initial value. 
However, when \cref{eq} only satisfies the condition of stiff 
well-posedness, one does not even know whether $G(t, \xi, \eta)$ is 
bounded, let alone control the high-frequency error.

We solve this problem indirectly, by first establishing the 
uniform quasi-stability (UQS) of the matrix set $\cF_1$, which is 
just the first conclusion of \Cref{thm}.
We aim to prove that the matrix exponential $\exp(H_{11} t)$ is 
uniformly bounded.
While \cref{eq:duhamel_u,eq:duhamel_v}, in accordance with stiff 
well-posedness, can be controlled uniformly for $t \ge 0$, 
$\xi \in \bR^d$, and $\eta \ge 0$, these alone are insufficient to 
guarantee the uniform boundedness of the integral term of 
\cref{eq:duhamel_u}.
It is important to note that the first right-hand term of 
\cref{eq:duhamel_u} is independent of $\eta$. 
The key in the proof is to examine the solution with zero initial 
value of fast variables and to show the integral term of 
\cref{eq:duhamel_u} tends to zero for fixed $t \ge 0$ and 
$\xi \in \bR^d$ as $\eta$ tends to infinity.

The asymptotic behavior of the integral term of \cref{eq:duhamel_v} 
is similar to our generalized Riemann--Lebesgue lemma in 
\Cref{lemma:gRL}. 
However, the integrand $\hat{u}$ depends on $\hat{v}$ according to 
\cref{eq:duhamel_u}, which brings additional challenges. 
Therefore, we utilize \cref{eq:duhamel_u,eq:duhamel_v} to construct 
a Picard sequence. By carefully choosing the Banach space and its 
subset being considered, we can prove that the iteration formula
producing the Picard sequence is determined by a contraction mapping,
which allows us to utilize the Banach fixed-point theorem to 
approximate the solution and to estimate it. 
In this way, one can obtain the uniform quasi-stability of the matrix 
set $\cF_1$ as summarized in \Cref{lemma:UQS_of_F1}, which estimates 
the error at the high-frequency region in 
\Cref{lemma:high_frequency}. 
Together with \Cref{lemma:low_frequency}, one can complete the proof 
of \Cref{thm}.

Let us provide a visual flowchart in \Cref{fig} to summarize the 
previous discussions and illustrate the main logical progression 
in our proof.

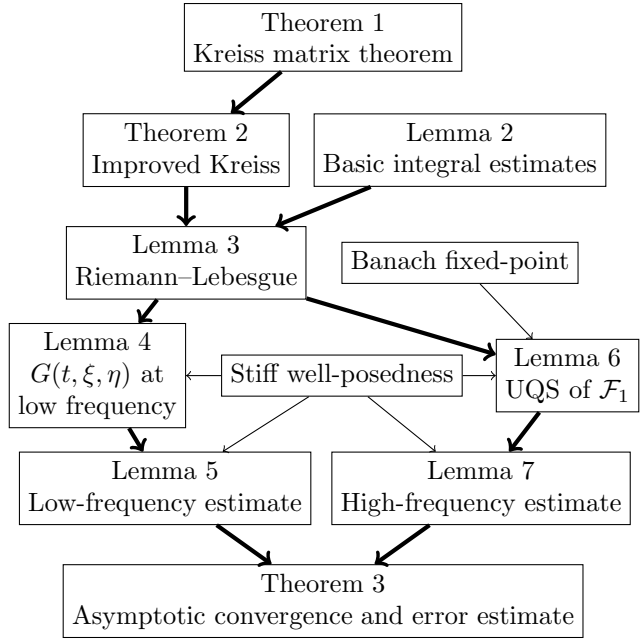
\begin{figure}[h]
  \centering
  \begin{tikzpicture}
    \node (KMT) at (0, 7.5) 
      [draw, rectangle, align=center] 
      {\Cref{thm:KMT} \\ Kreiss matrix theorem};
    \node (BI) at (1.82, 6) 
      [draw, rectangle, align=center] 
      {\Cref{lemma:integral} \\ Basic integral estimates};
    \node (IKMT) at (-1.82, 6) 
      [draw, rectangle, align=center] 
      {\Cref{thm:IKMT} \\ Improved Kreiss};
    \node (gRL) at (-1.82, 4.5) 
      [draw, rectangle, align=center] 
      {\Cref{lemma:gRL} \\ Riemann--Lebesgue};
    \node (banach) at (1.82, 4.5) 
      [draw, rectangle, align=center] 
      {Banach fixed-point};
    \node (low_frequency_G) at (-3, 3) 
      [draw, rectangle, align=center] 
      {\Cref{lemma:low_frequency_G} \\ $G(t, \xi, \eta)$ at \\ 
      low frequency};
    \node (UQS_of_F1) at (3.25, 3) 
      [draw, rectangle, align=center] 
      {\Cref{lemma:UQS_of_F1} \\ UQS of $\cF_1$};
    \node (stiff_well_posedness) at (0.25, 3) 
      [draw, rectangle] 
      {Stiff well-posedness};
    \node (low_frequency) at (-2.1, 1.5) 
      [draw, rectangle, align=center] 
      {\Cref{lemma:low_frequency} \\ Low-frequency estimate};
    \node (high_frequency) at (2.1, 1.5) 
      [draw, rectangle, align=center] 
      {\Cref{lemma:high_frequency} \\ High-frequency estimate};
    \node (thm) at (0, 0) 
      [draw, rectangle, align=center] 
      {\Cref{thm} \\ Asymptotic convergence and error estimate};
    
    \draw[->,ultra thick] (KMT) -- (IKMT);
    \draw[->,ultra thick] (IKMT) -- (gRL);
    \draw[->,ultra thick] (BI) -- (gRL);
    \draw[->,ultra thick] (gRL) -- (low_frequency_G);
    \draw[->,ultra thick] (gRL) -- (UQS_of_F1);
    \draw[->] (banach) -- (UQS_of_F1);
    \draw[->] (stiff_well_posedness) -- (low_frequency_G);
    \draw[->] (stiff_well_posedness) -- (UQS_of_F1);
    \draw[->,ultra thick] (UQS_of_F1) -- (high_frequency);
    \draw[->,ultra thick] (low_frequency_G) -- (low_frequency);
    \draw[->] (stiff_well_posedness) -- (low_frequency);
    \draw[->] (stiff_well_posedness) -- (high_frequency);
    \draw[->,ultra thick] (low_frequency) -- (thm);
    \draw[->,ultra thick] (high_frequency) -- (thm);
    \end{tikzpicture}
  \caption{Flowchart of the proof}
  \label{fig}
\end{figure}

\subsection{Basic estimates of an elementary integral}
Let us begin with basic estimates of the following integral:
\begin{equation}\label{equ:integral}
  \cI(\alpha_1, \alpha_2, \beta_1, \beta_2) := 
  \int_{- \infty}^{+ \infty} \frac{1}{|y - \alpha_1| + \beta_1} 
  \cdot \frac{1}{|y - \alpha_2| + \beta_2} \,\rd y.
\end{equation}
Here, $\alpha_j \in \bR$ and $\beta_j > 0$ for $j = 1, 2$.
Integrals of this form will be used several times in our proof in 
conjunction with \Cref{thm:IKMT}. While this integral has an explicit 
analytical expression, it is complicated when 
$\beta_1 \ne \beta_2$. Therefore, let us give some estimates.

\begin{lemma}\label{lemma:integral}
  The following estimates of the integral defined in 
  \cref{equ:integral} hold:
  \begin{enumerate}
    \item If $\beta_1 = \beta_2 = \beta$, then 
    \begin{displaymath}
      \cI(\alpha_1, \alpha_2, \beta, \beta) = 
      \begin{cases}
        \frac{4 (d + 1) \log(d + 1)}{\beta d (d + 2)} 
        < \frac{2}{\beta}, & \quad \text{ if } d > 0, \\
        \frac{2}{\beta}, & \quad \text{ if } d = 0,
      \end{cases}
    \end{displaymath}
    where $d = \frac{|\alpha_1 - \alpha_2|}{\beta} \ge 0$;
    \item If $\beta_1 \ne \beta_2$, then 
    \begin{displaymath}
      \cI(\alpha_1, \alpha_2, \beta_1, \beta_2) \le 
      2 \frac{\log \beta_1 - \log \beta_2}{\beta_1 - \beta_2} <
      \frac{2}{\min \left\{ \beta_1, \beta_2 \right\}}.
    \end{displaymath}
  \end{enumerate}
\end{lemma}

\begin{proof}
  The first conclusion is straightforward by direct calculation and 
  by noticing that the function 
  $f(d) := \frac{4 (d + 1) \log (d + 1)}{d (d + 2)}$ is strictly 
  decreasing on $(0, + \infty)$, and it has the limit 
  $\lim_{d \rightarrow 0} f(d) = 2$.

  As for the second conclusion, we notice that 
  \begin{displaymath}
    \int_{- \infty}^{+ \infty} 
    \left( \frac{1}{|y - \alpha_1| + \beta_1} 
    - \frac{1}{|y - \alpha_2| + \beta_1} \right) \cdot 
    \left( \frac{1}{|y - \alpha_1| + \beta_2} 
    - \frac{1}{|y - \alpha_2| + \beta_2} \right) \,\rd y \ge 0,
  \end{displaymath}
  and that 
  \begin{displaymath}
    \begin{aligned}
      & \int_{- \infty}^{+ \infty} \frac{1}{|y - \alpha_1| + \beta_1} 
      \left( \frac{1}{|y - \alpha_1| + \beta_2} - 
      \frac{1}{|y - \alpha_2| + \beta_2} \right) \,\rd y \\
      = & \int_{- \infty}^{+ \infty}
      \frac{1}{|y - \alpha_2| + \beta_1} 
      \left( \frac{1}{|y - \alpha_2| + \beta_2} - 
      \frac{1}{|y - \alpha_1| + \beta_2} \right) \,\rd y.
    \end{aligned}
  \end{displaymath}
  Therefore, 
  \begin{displaymath}
    \cI(\alpha_1, \alpha_2, \beta_1, \beta_2) \le 
    \cI(\alpha_1, \alpha_1, \beta_1, \beta_2) = 
    2 \frac{\log \beta_1 - \log \beta_2}{\beta_1 - \beta_2} <
    \frac{2}{\min \left\{ \beta_1, \beta_2 \right\}},
  \end{displaymath}
  by direct calculation.
\end{proof}

\subsection{Generalized Riemann--Lebesgue lemma}
Before delving into the exploration of several nontrivial corollaries 
of stiff well-posedness, we need a somehow generalization of 
Riemann--Lebesgue lemma, 
which is an application of \Cref{thm:IKMT}. 
The integral in this version of the Riemann--Lebesgue type estimate 
contains a matrix exponential, which, intuitively speaking, allows 
both oscillation and dissipation.

\begin{lemma}\label{lemma:gRL}
  Assume that the $r \times r$ matrix $B$ is invertible and 
  quasi-stable, the $r \times r$ matrix $M$ is arbitrary,
  and the function $f \in C^1([0, T])$ for a fixed $T > 0$.
  The following estimate holds for 
  $\eta \ge 6 \gamma \delta^{-1}$:
  \begin{displaymath}
    \bigg\| \int_0^T \exp \big( (\eta B + M) (T - s) \big) f(s) 
    \,\rd s \bigg\| \le \frac{32 \sqrt{2} r K e^{\gamma T}}
    {\pi \eta \delta} \|f\|_{C^1([0, T])} \log 
    \frac{\eta \delta + \gamma}{\gamma},
  \end{displaymath}
  where $K$ and $\delta$ depend on the matrix $B$, and 
  $\gamma = 2 K \norm{M} + 1$.
  Furthermore, if the matrix $M$ and the function 
  $f \in C^1([0, +\infty))$ take the form of 
  \begin{displaymath}
    M = M(\theta) = \sum_{j = 1}^d \theta_j M_j, \quad 
    f = f(\cdot; \theta),
  \end{displaymath}
  respectively, where $\theta \in \bR^d$, the matrices $M_j$ with 
  $j = 1, 2, \ldots, d$ are chosen such that the set of matrices,
  \begin{displaymath}
    \cF = \left\{ \eta B + M(\theta) \,\big|\, 
    \theta \in \bR^d, \eta \ge 0 \right\},
  \end{displaymath}
  is uniformly quasi-stable, 
  and the function $f(\cdot; \theta)$ satisfies 
  \begin{displaymath}
    f(\cdot; \omega \theta) = f(\omega \cdot; \theta),
  \end{displaymath}
  for each $\omega > 0$,
  then the following estimate holds for $\eta >\beta (|\theta| + 1)$:
  \begin{multline*}
    \bigg\| \int_0^T \exp \big( (\eta B + M(\theta)) (T - s) \big) 
    f(s; \theta) \,\rd s \bigg\| \\ \le
    \frac{16 \sqrt{2} r \tilde{K} e^T}{\pi \eta \tilde{\delta}} 
    \sup_{|\theta| = 1} \|f(\cdot; \theta)\|_{C^1([0, T])} 
    \log(\eta \tilde{\delta} + 1), 
  \end{multline*}
  where $\tilde{K} > 0$ and $\beta > 0$ depend on the matrix 
  set $\cF$, and $\tilde{\delta} > 0$ depends on the matrix $B$.
\end{lemma}

\begin{proof}
  For a complex number $z \in \bC$, let us denote its real part and 
  imaginary part by $\gamma = \Re(z)$ and $y = \Im(z)$, respectively, 
  i.e., $z = \gamma + i y$ with $\gamma, y \in \bR$.
  We consider the zero extension of the function $f$, i.e., we define 
  the value of $f(t)$ as $f(t) \equiv 0$ for each $t > T$, and we 
  still use the notation $f(t)$.
  The Laplace transform $\tilde{f}(z)$ of $f(t)$ is given as follows,
  \begin{multline*}
    \tilde{f}(z) := \int_0^T e^{- z t} f(t) \,\rd t = 
    \int_0^T e^{- i t y} e^{- \gamma t} f(t) \,\rd t \\
    = \left( \frac{1}{- i y} e^{-i t y} e^{- \gamma t} f(t) \right)
    \bigg|_{t = 0}^{t = T} - \int_0^T \frac{1}{- i y} e^{- i t y}
    \od{}{t} \left( e^{- \gamma t} f(t) \right) \,\rd t,
  \end{multline*}
  which yields the following estimate of $\tilde{f}(z)$:
  \begin{multline*}
    |\tilde{f}(z)| \le \min \left\{ \frac{1}{\gamma}, 
    \frac{2 + \gamma^{-1} + e^{- \gamma T}}{|y|} \right\} 
    \|f\|_{C^1([0,T])} \\ 
    \le \min \left\{ \frac{1}{\gamma}, \frac{4}{|y|} \right\}
    \|f\|_{C^1([0,T])}
    \le \frac{8 \|f\|_{C^1([0,T])}}{|y| + 4 \gamma},
  \end{multline*}
  where the constant $\gamma > 1$ is to be determined.
  The Laplace transform $\tilde{R}(z)$ of the function $R(t)$ defined 
  as 
  \begin{displaymath}
    R(t) = \int_0^t \exp \big( (\eta B + M)(t - s) \big) f(s) \,\rd s,
  \end{displaymath}
  is given by 
  \begin{displaymath}
    \tilde{R}(z) = \int_0^{+\infty} e^{- z t} R(t) \,\rd t = 
    (z I - \eta B - M)^{-1} \tilde{f}(z).
  \end{displaymath}
  By the inverse Laplace transform, one can obtain that 
  \begin{multline*}
    R(t) = \frac{1}{2 \pi i} 
    \int_{\gamma - i \infty}^{\gamma + i \infty} e^{t z} 
    \tilde{R}(z) \,\rd z 
    = \frac{e^{\gamma t}}{2 \pi} \int_{-\infty}^{+\infty}  
    e^{i t y} (z I - \eta B - M)^{-1} \tilde{f}(z) \,\rd y \\
    = \frac{e^{\gamma t}}{2 \pi} \int_{-\infty}^{+\infty} 
    e^{i t y} \big( I - (z I - \eta B)^{-1} M \big)^{-1} 
    (z I - \eta B)^{-1} \tilde{f}(z) \,\rd y.
  \end{multline*}
  Since the matrix $B$ is quasi-stable, the set of matrices $\eta B$ 
  with $\eta \ge 0$ is uniformly quasi-stable. By \Cref{thm:IKMT}, 
  there exists a constant $K > 0$ such that 
  \begin{displaymath}
    \|(z I - \eta B)^{-1}\| \le K \max_{\lambda \in \sigma(B)} 
    |z - \eta \lambda|^{-1} \le \frac{K}{\gamma}.
  \end{displaymath}
  By taking $\gamma = 2 K \|M\| + 1$, one can obtain that 
  \begin{displaymath}
    \|(z I - \eta B)^{-1} M\| \le \frac{K \|M\|}{\gamma} < \frac12,
  \end{displaymath}
  and thus, 
  \begin{displaymath}
    \big\| \big( I - (z I - \eta B)^{-1} M \big)^{-1} \big\| < 2.
  \end{displaymath}
  Therefore, 
  \begin{multline*}
    \|R(T)\| \le \frac{e^{\gamma T}}{2 \pi} \int_{-\infty}^{+\infty}
    \big\| \big( I - (z I - \eta B)^{-1} M \big)^{-1} \big\| \cdot 
    \|(z I - \eta B)^{-1}\| \cdot |\tilde{f}(z)| \,\rd y \\ 
    \le \frac{K e^{\gamma T}}{\pi}
    \|f\|_{C^1([0, T])} \int_{-\infty}^{+\infty} 
    \max_{\lambda \in \sigma(B)} |z - \eta \lambda|^{-1}
    \frac{8}{|y| + 4 \gamma} \,\rd y \\ 
    \le \frac{8 \sqrt{2} K e^{\gamma T}}{\pi} 
    \|f\|_{C^1([0, T])} 
    \sum_{\lambda \in \sigma(B)} \int_{-\infty}^{+\infty} 
    \frac{1}{|y - \eta \Im(\lambda)| + \gamma - \eta \Re(\lambda)}
    \cdot \frac{1}{|y| + 4 \gamma} \,\rd y \\ 
    = \frac{8 \sqrt{2} K e^{\gamma T}}{\pi} 
    \|f\|_{C^1([0, T])} \sum_{\lambda \in \sigma(B)} 
    \cI(\eta \Im(\lambda), 0, \gamma - \eta \Re(\lambda), 4 \gamma).
  \end{multline*}
  Note that the matrix $B$ is invertible and quasi-stable, there 
  exists a constant $\delta > 0$ dependent on $B$ such that 
  for each eigenvalue $\lambda \in \sigma(B)$ either 
  $\Re(\lambda) \le - \delta$ or $|\Im(\lambda)| \ge \delta$ holds.
  If $\Re(\lambda) \le - \delta$, then 
  \begin{displaymath}
    \cI(\eta \Im(\lambda), 0, \gamma - \eta \Re(\lambda), 4 \gamma) 
    \le \cI(\eta \Im(\lambda), 0, \gamma + \eta \delta, 4 \gamma)
    \le \frac{2}{\eta \delta - 3 \gamma} \log 
    \frac{\eta \delta + \gamma}{4 \gamma}.
  \end{displaymath}
  If $|\Im(\lambda)| \ge \delta$, then 
  \begin{displaymath}
    \cI(\eta \Im(\lambda), 0, \gamma - \eta \Re(\lambda), 4 \gamma) 
    \le \cI(\eta \delta, 0, \gamma, \gamma) \le 
    \frac{4 (\eta \delta + \gamma)}
    {\eta \delta (\eta \delta + 2 \gamma)} 
    \log \frac{\eta \delta + \gamma}{\gamma}.
  \end{displaymath}
  Both cases yield that 
  \begin{displaymath}
    \cI(\eta \Im(\lambda), 0, \gamma - \eta \Re(\lambda), 4 \gamma) 
    \le \frac{4}{\eta \delta} 
    \log \frac{\eta \delta + \gamma}{\gamma},
  \end{displaymath}
  as long as $\eta \delta \ge 6 \gamma$, which completes 
  the proof of the first conclusion of this lemma.

  Let us now prove the second conclusion of this lemma. 
  The Laplace transform of $f(t; \theta)$ satisfies the following 
  scaling property:
  \begin{displaymath}
    \tilde{f}(z; \omega \theta) = \omega^{-1}
    \tilde{f}(\omega^{-1} z; \theta), 
  \end{displaymath}
  for each $\omega > 0$, which yields the following estimate:
  \begin{displaymath}
    |\tilde{f}(z; \theta)| \le \frac{8}{|y| + 4 \gamma} 
    \|f(\cdot; |\theta|^{-1} \theta)\|_{C^1([0, +\infty))}.
  \end{displaymath}
  By applying \Cref{thm:IKMT} on the uniformly quasi-stable matrix 
  set $\cF$ and using the previous arguments, one can obtain the 
  following estimate: 
  \begin{multline*}
    \|R(T; \theta)\| \le \frac{4 \sqrt{2} \tilde{K} 
    e^{\tilde{\gamma} T}}{\pi}
    \sup_{|\theta| = 1} \|f(\cdot; \theta)\|_{C^1([0, +\infty))} \\
    \sum_{\lambda \in \sigma(\eta B + M(\theta))} 
    \cI(\Im(\lambda), 0, \tilde{\gamma} - \Re(\lambda), 
    4 \tilde{\gamma}),
  \end{multline*}
  where $\tilde{\gamma} > 0$ is arbitrary and to be determined, 
  and $\tilde{K} > 0$ depends on the matrix set $\cF$.
  Since the matrix $B$ is invertible, there exist constants 
  $\tilde{\alpha} > 0$ and $\tilde{\delta} > 0$ dependent on the 
  matrix set $\cF$ such that if $|\theta| \le \tilde{\alpha} \eta$, 
  then for each eigenvalue $\lambda \in \sigma(\eta B + M(\theta))$ 
  either $\Re(\lambda) \le - \eta \tilde{\delta}$ or 
  $|\Im(\lambda)| \ge \eta \tilde{\delta}$ holds.
  If $\Re(\lambda) \le - \eta \tilde{\delta}$, then 
  \begin{displaymath}
    \cI(\Im(\lambda), 0, \tilde{\gamma} - \Re(\lambda), 
    4 \tilde{\gamma}) \le 
    \cI(\Im(\lambda), 0, \tilde{\gamma} + \eta \tilde{\delta}, 
    4 \tilde{\gamma}) \le 
    \frac{2}{\eta \tilde{\delta} - 3 \tilde{\gamma}} 
    \log 
    \frac{\eta \tilde{\delta} + \tilde{\gamma}}{4 \tilde{\gamma}}.
  \end{displaymath}
  If $|\Im(\lambda)| \ge \eta \tilde{\delta}$, then 
  \begin{displaymath}
    \cI(\Im(\lambda), 0, \tilde{\gamma} - \Re(\lambda), 
    4 \tilde{\gamma}) \le 
    \cI(\eta \tilde{\delta}, 0, \tilde{\gamma}, \tilde{\gamma}) \le 
    \frac{4 (\eta \tilde{\delta} + \tilde{\gamma})}
    {\eta \tilde{\delta} (\eta \tilde{\delta} + 2 \tilde{\gamma})} 
    \log \frac{\eta \tilde{\delta} + \tilde{\gamma}}{\tilde{\gamma}}.
  \end{displaymath}
  Both cases yield that 
  \begin{displaymath}
    \cI(\Im(\lambda), 0, \tilde{\gamma} - \Re(\lambda), 
    \tilde{\alpha}_2) \le \frac{4}{\eta \tilde{\delta}} 
    \log \frac{\eta \tilde{\delta} + \tilde{\gamma}}{\tilde{\gamma}},
  \end{displaymath}
  as long as 
  $\eta \tilde{\delta} \ge 6 \tilde{\gamma}$.
  One can choose $\tilde{\gamma} = 1$, and then obtain the desired 
  result. 
\end{proof}

\subsection{Low-frequency estimate}
We would like to estimate the key term $G(t, \xi, \eta)$ in the 
Mori--Zwanzig technique presented in \cref{eq:MZ} at the 
low-frequency region by applying \Cref{lemma:gRL}, and then obtain 
the low-frequency estimate.

\begin{lemma}\label{lemma:low_frequency_G}
  If the matrix $B$ is quasi-stable, the following estimate of 
  $G(t, \xi, \eta)$ holds for $\eta > \beta \gamma$:
  \begin{displaymath}
    \|G(t, \xi, \eta)\| \le \frac{C |\xi| e^{\gamma t}}{\eta} 
    \|\exp(H_{11} \cdot)\|_{C^1([0, t])} 
    \log \frac{\eta \delta + \gamma}{\gamma},
  \end{displaymath}
  where $\gamma = \alpha |\xi| + 1$, and the positive constants 
  $C$, $\delta$, $\alpha$ and $\beta$ depend on 
  the matrices $A^j$ and $Q$.
  Furthermore, if the matrix sets $\cF_1$ and $\cF_2$ defined in 
  \cref{eq:F1,eq:F2}, respectively, are uniformly quasi-stable, then 
  the following estimate of $G(t, \xi, \eta)$ holds for 
  $\eta > \tilde{\beta}(|\xi| + 1)$: 
  \begin{displaymath}
    \|G(t, \xi, \eta)\| \le \frac{\tilde{C} |\xi| e^t}{\eta} 
    \log(\eta \tilde{\delta} + 1),
  \end{displaymath}
  where the positive constants $\tilde{C}$, $\tilde{\delta}$ and 
  $\tilde{\beta}$ depend on the matrices $A^j$ and $Q$.
\end{lemma}

\begin{proof}
  This is a direct corollary of \Cref{lemma:gRL}.
\end{proof}

\begin{remark}
  The second conclusion in \Cref{lemma:gRL} is crucial for the 
  estimation of $G(t, \xi, \eta)$ that depends linearly on 
  $|\xi|$ and vanishes as $\eta$ tends to infinity. This aspect is 
  essential to derive the corresponding error estimate in \Cref{thm}.
\end{remark}

\begin{lemma}\label{lemma:low_frequency}
  If \cref{eq} is stiffly well-posed, the following convergence holds 
  for each $t > 0$ and $\Xi > 0$:
  \begin{displaymath}
    \lim \limits_{\eta \to +\infty}
    \norm{ \big( \hat{u}(t, \cdot, \eta) - \exp(H_{11} t) \hat{u}_0
    \big) \mathbbm{1}_{|\cdot| \le \Xi} }_{L^2} = 0.
  \end{displaymath}
  Furthermore, if the matrix sets $\cF_1$ and $\cF_2$ are uniformly 
  quasi-stable, and the initial data satisfies $U_0 \in H^2$, 
  then there exist positive constants $C$, $\tilde{\delta}$ and 
  $\tilde{\beta}$ dependent on the matrices $A_j$ and $Q$ such that 
  the following estimate holds for each $t > 0$ and 
  $\eta > 2 \tilde{\beta}$:
  \begin{equation}\label{eq:low_frequency_error}
    \norm{ \big( \hat{u}(t, \cdot, \eta) - \exp(H_{11} t) \hat{u}_0
    \big) \mathbbm{1}_{|\cdot| \le \Xi(\eta)} }_{L^2} \le 
    \frac{C e^{t}}{\eta} \|U_0\|_{H^2}
    \log (\eta \tilde{\delta} + 1),
  \end{equation}
  where $\Xi(\eta) = \tilde{\beta}^{-1} \eta - 1$.
\end{lemma}

\begin{proof}
  By \Cref{lemma:low_frequency_G}, the following convergence holds:
  \begin{displaymath}
    \lim \limits_{\eta \to +\infty} \sup_{s \in [0, t]} 
    \sup_{|\xi| \le \Xi} \| G(s, \xi, \eta) \| = 0.
  \end{displaymath}
  Due to the stiff well-posedness of \cref{eq}, there exists a 
  constant $K_1 > 0$ such that 
  \begin{displaymath}
    \|{\hat{u}(s, \xi, \eta)}\| \le K_1 \|{\hat{U}_0(\xi)}\|,
  \end{displaymath}
  for each $s \ge 0$, $\xi \in \bR^d$ and $\eta \ge 0$. Therefore, 
  \begin{multline*}
    \norm{ \big( \hat{u}(t, \cdot, \eta) - \exp(H_{11} t) \hat{u}_0
    \big) \mathbbm{1}_{|\cdot| \le \Xi} }_{L^2} \\
    \le \norm{ G(t, \cdot, \eta) \hat{v}_0 
    \mathbbm{1}_{|\cdot| \le \Xi} }_{L^2} 
    + \int_0^t 
    \norm{ G(t-s, \cdot, \eta) H_{21} \hat{u}(s, \cdot, \eta)
    \mathbbm{1}_{|\cdot| \le \Xi} }_{L^2} \,\rd s \\
    \le \sup_{s \in [0, t]} \sup_{|\xi| \le \Xi} \| G(s, \xi, \eta) \|
    \cdot \bigg( \| \hat{v}_0 \|_{L^2} + 
    t K_1 \sup_{|\xi| \le \Xi} \| H_{21}(\xi, 0) \| \cdot
    \| \hat{U}_0 \|_{L^2} \bigg),
  \end{multline*}
  which yields the desired convergence.

  Let us now consider the case when the matrix sets $\cF_1$ and 
  $\cF_2$ are uniformly quasi-stable.
  Note that 
  \begin{multline*}
    \norm{ \big( \hat{u}(t, \cdot, \eta) - \exp(H_{11} t) \hat{u}_0
    \big) \mathbbm{1}_{|\cdot| \le \Xi(\eta)} }_{L^2} \\
    \le \norm{ G(t, \cdot, \eta) \hat{v}_0 
    \mathbbm{1}_{|\cdot| \le \Xi(\eta)} }_{L^2} 
    + \int_0^t 
    \norm{ G(t-s, \cdot, \eta) H_{21} \hat{u}(s, \cdot, \eta)
    \mathbbm{1}_{|\cdot| \le \Xi(\eta)} }_{L^2}
    \,\rd s.
  \end{multline*}
  By \Cref{lemma:low_frequency_G},
  the first right-hand term in this inequality can be bounded by 
  \begin{displaymath}
    \norm{ G(t, \cdot, \eta) \hat{v}_0 
    \mathbbm{1}_{|\cdot| \le \Xi(\eta)} }_{L^2} \le 
    \frac{\tilde{C} e^t}{\eta} \|v_0\|_{H^1}
    \log(\eta \tilde{\delta} + 1),
  \end{displaymath}
  and the integrand in the integral term of this inequality can be 
  estimated as follows:
  \begin{displaymath}
    \norm{ G(t-s, \cdot, \eta) H_{21} \hat{u}(s, \cdot, \eta)
    \mathbbm{1}_{|\cdot| \le \Xi(\eta)} }_{L^2} \le 
    \frac{\tilde{C} e^{t - s}}{\eta} \|U_0\|_{H^2}
    \log(\eta \tilde{\delta} + 1),
  \end{displaymath}
  where the positive constants $\tilde{C}$ and $\tilde{\delta}$ 
  depend on the matrices $A_j$ and $Q$.
  These estimates complete the proof of this lemma.
\end{proof}

\subsection{Strong hyperbolicity of reduced systems}
Let us investigate a nontrivial corollary of stiff 
well-posedness. 
The following lemma shows that stiff well-posedness is sufficient 
to guarantee the strong hyperbolicity of reduced systems. 

\begin{lemma}\label{lemma:UQS_of_F1}
  If \cref{eq} is stiffly well-posed, then the set $\cF_1$ 
  of matrices defined in \cref{eq:F1} is uniformly quasi-stable.
\end{lemma}

\begin{proof}
  Let us take fixed $\xi \in \bR^d$ and $t > 0$. Our aim is to  
  prove that the matrix exponential $\|\exp(H_{11}(\xi, 0) t)\|$ can 
  be bounded from above by a constant independent of $\xi$ and $t$.
  For each $\tilde{u}_0 \in \bR^{n - r}$ and $T > 0$,
  let us define the metric space $X$ as follows, 
  \begin{multline*}
    X := \Big\{ (\tilde{u}, \tilde{v}) \in C^1([0, T], \bR^{n - r}) 
    \times C^0([0, T], \bR^r) \,\Big|\, \\
    \tilde{u}(0) = \tilde{u}_0, \tilde{v}(0) = 0,
    \|\tilde{u}\|_{C^1} \le \Lambda_1, 
    \|\tilde{v}\|_{C^0} \le \Lambda_2 \Big\},
  \end{multline*}
  with the metric $\rho$ defined as 
  \begin{displaymath}
    \rho \big( (\tilde{u}_1, \tilde{v}_1), (\tilde{u}_2, \tilde{v}_2) 
    \big) = \alpha_1 
    \|\tilde{u}_1 - \tilde{u}_2\|_{C^1} + 
    \alpha_2 \|\tilde{v}_1 - \tilde{v}_2\|_{C^0},
  \end{displaymath}
  where the positive constants $\Lambda_1$, $\Lambda_2$, $\alpha_1$, 
  and $\alpha_2$ are to be determined.
  Let us construct the following mapping, 
  \begin{displaymath}
    \cT : (\tilde{u}, \tilde{v}) \in X \mapsto 
    (\cT_1(\tilde{u}, \tilde{v}), \cT_2(\tilde{u}, \tilde{v})) \in 
    C^1([0, T], \bR^{n - r}) \times C^0([0, T], \bR^r),
  \end{displaymath}
  where 
  \begin{gather*}
    \cT_1(\tilde{u}, \tilde{v})(t) := \exp(H_{11} t) \tilde{u}_0 + 
    \int_0^t \exp(H_{11} (t - s)) H_{12} \tilde{v}(s) \,\rd s, \\
    \cT_2(\tilde{u}, \tilde{v})(t) := \int_0^t \exp(H_{22} (t - s)) 
    H_{21} \tilde{u}(s) \,\rd s.
  \end{gather*}
  Let us first prove that $\cT(\tilde{u}, \tilde{v}) \in X$ whenever 
  $(\tilde{u}, \tilde{v}) \in X$.
  By \Cref{lemma:gRL}, one can obtain that 
  \begin{displaymath}
    \|\cT_2(\tilde{u}, \tilde{v})\|_{C^0} \le C 
    \eta^{-1} |\log \eta| \cdot 
    \|\tilde{u}\|_{C^1},
  \end{displaymath}
  for sufficiently large $\eta$. By direct calculation, one can 
  obtain that 
  \begin{displaymath}
    \|\cT_1(\tilde{u}, \tilde{v})\|_{C^1} \le  
    C (\|\tilde{u}_0\| + \|\tilde{v}\|_{C^0}).
  \end{displaymath}
  Here the constant $C$ is dependent on $t$, $\xi$, $A_j$ and $Q$, 
  but is independent of $\eta$.
  As long as 
  \begin{displaymath}
    C \eta^{-1} |\log \eta| \cdot \Lambda_1 \le \Lambda_2, \quad 
    C (\|u_0\| + \Lambda_2) \le \Lambda_1,
  \end{displaymath}
  the mapping $\cT$ maps $X$ onto itself. Similarly, note that 
  \begin{multline*}
    \rho \big( \cT(\tilde{u}_1, \tilde{v}_1), 
    \cT(\tilde{u}_2, \tilde{v}_2) \big) \\
    = \alpha_1 \| \cT_1(\tilde{u}_1, \tilde{v}_1) - 
    \cT_1(\tilde{u}_2, \tilde{v}_2) \|_{C^1} 
    + \alpha_2 \| \cT_2(\tilde{u}_1, \tilde{v}_1) - 
    \cT_2(\tilde{u}_2, \tilde{v}_2) \|_{C^0} \\
    \le \alpha_1 C 
    \| \tilde{v}_1 - \tilde{v}_2 \|_{C^0} + 
    \alpha_2 C \eta^{-1} |\log \eta| \cdot 
    \| \tilde{u}_1 - \tilde{u}_2 \|_{C^1}  \\ 
    \le \frac{1}{2} \alpha_1  
    \| \tilde{u}_1 - \tilde{u}_2 \|_{C^1} + 
    \frac{1}{2} \alpha_2 
    \| \tilde{v}_1 - \tilde{v}_2 \|_{C^0} 
    = \frac{1}{2} \rho \big( (\tilde{u}_1, \tilde{v}_1), 
    (\tilde{u}_2, \tilde{v}_2) \big),
  \end{multline*}
  as long as 
  \begin{displaymath}
    \alpha_1 C \le \frac{1}{2} \alpha_2, \quad 
    \alpha_2 C \eta^{-1} |\log \eta| \le \frac{1}{2} \alpha_1.
  \end{displaymath}
  At this time, the mapping $\cT$ is a contraction mapping on $X$.
  For example, one can choose sufficiently large $\eta$ and 
  the constants $\Lambda_1$, $\Lambda_2$, $\alpha_1$ and $\alpha_2$ 
  given by
  \begin{displaymath}
    \Lambda_1 = 2 C \norm{\tilde{u}_0}, \quad 
    \Lambda_2 = \norm{\tilde{u}_0}, \quad 
    \alpha_1 = 1, \quad \alpha_2 = 2 C,
  \end{displaymath}
  such that all the requirements mentioned previously are satisfied.
  Let us construct the Picard sequence as follows,
  \begin{gather*}
    \tilde{u}_0(t) \equiv \tilde{u}_0, \quad \tilde{v}_0(t) \equiv 0, 
    \quad \forall t \in [0, T], \\
    (\tilde{u}_{k + 1}, \tilde{v}_{k + 1}) = 
    \cT(\tilde{u}_k, \tilde{v}_k), \quad \forall k \in \bN.
  \end{gather*}
  By the Banach fixed-point theorem, there exists 
  $(\tilde{u}_*, \tilde{v}_*) \in X$ such that 
  \begin{displaymath}
    \lim\limits_{k \to \infty}
    \rho \big( (\tilde{u}_k, \tilde{v}_k), 
    (\tilde{u}_*, \tilde{v}_*) \big) = 0.
  \end{displaymath}
  At this time, $(\tilde{u}_*, \tilde{v}_*)$ satisfies the equations
  \begin{gather*}
    \tilde{u}_*(t) = \exp(H_{11} t) \tilde{u}_0 + \int_0^t 
    \exp(H_{11} (t - s)) H_{12} \tilde{v}_*(s) \,\rd s, \\
    \tilde{v}_*(t) = \int_0^t \exp(H_{22} (t - s)) H_{21} 
    \tilde{u}_*(s) \,\rd s.
  \end{gather*}
  Therefore, $(\tilde{u}_*, \tilde{v}_*)$ is the solution to 
  \begin{gather*}
    \od{\tilde{u}}{t}  = H_{11} \tilde{u} + H_{12} \tilde{v}, \quad 
    \tilde{u}(0) = \tilde{u}_0, \\
    \od{\tilde{v}}{t}  = H_{21} \tilde{u} + H_{22} \tilde{v}, \quad 
    \tilde{v}(0) = 0. 
  \end{gather*}
  According to the definition of stiff well-posedness, 
  there exists a positive constant $\tilde{C}$ independent of 
  $t$, $\xi$ and $\eta$ such that 
  \begin{displaymath}
    \| \tilde{u}(t) \| \le \tilde{C} \|\tilde{u}_0\|.
  \end{displaymath}
  Note that there exists a constant $C > 0$ independent of $\eta$ 
  such that  
  \begin{multline*}
    \nnorm{\int_0^t \exp(H_{11} (t - s)) H_{12} \tilde{v}_*(s) 
    \,\rd s}
    \le C \|{\tilde{v}_*}\|_{C^0} 
    \le C^2 \eta^{-1} |\log \eta| \cdot 
    \|\tilde{u}_*\|_{C^1} \\
    \le C^2 \eta^{-1} |\log \eta| \cdot \Lambda_1 
    = 2 C^3 \eta^{-1} |\log \eta| \cdot \|\tilde{u}_0\|.
  \end{multline*}
  By choosing sufficiently large $\eta$, one can obtain that 
  \begin{displaymath}
    \|\exp(H_{11} t) \tilde{u}_0\| \le 2 \tilde{C} \|\tilde{u}_0\|,
  \end{displaymath}
  where $\tilde{C}$ is independent of $t$ and $\xi$.
\end{proof}

\subsection{High-frequency estimate}

Let us now consider the high-frequency estimate based on 
\Cref{lemma:UQS_of_F1}.

\begin{lemma}\label{lemma:high_frequency}
  If \cref{eq} is stiffly well-posed, the following convergence holds 
  for each $t > 0$: 
  \begin{displaymath}
    \lim_{\Xi \to +\infty} \sup_{t \ge 0} \sup_{\eta \ge 0}
    \norm{ \big( \hat{u}(t, \cdot, \eta) - \exp(H_{11} t) \hat{u}_0
    \big) \mathbbm{1}_{|\cdot| > \Xi} }_{L^2} = 0.
  \end{displaymath}
  Furthermore, if the initial data satisfies $U_0 \in H^s$ 
  with $s > 0$, then the following estimate holds:
  \begin{displaymath}
    \sup_{t \ge 0} \sup_{\eta \ge 0}
    \norm{ \big( \hat{u}(t, \cdot, \eta) - \exp(H_{11} t) \hat{u}_0
    \big) \mathbbm{1}_{|\cdot| > \Xi} }_{L^2} \le 
    \frac{C}{1 + \Xi^s} \|U_0\|_{H^s}.
  \end{displaymath}
\end{lemma}

\begin{proof}
  Due to the uniform quasi-stability of the matrix sets 
  $\cF_0$ and $\cF_1$, there exists a constant $C > 0$ independent of 
  $t$, $\xi$ and $\eta$ such that 
  \begin{displaymath}
    \|\hat{u}(t, \xi, \eta) - \exp(H_{11} t) \hat{u}_0(\xi)\| \le 
    C \|\hat{U}_0(\xi)\|.
  \end{displaymath}
  Therefore, 
  \begin{displaymath}
    \sup_{t \ge 0} \sup_{\eta \ge 0}
    \norm{ \big( \hat{u}(t, \cdot, \eta) - \exp(H_{11} t) \hat{u}_0
    \big) \mathbbm{1}_{|\cdot| > \Xi} }_{L^2} \le C 
    \| \hat{U}_0 \mathbbm{1}_{|\cdot| > \Xi} \|_{L^2},
  \end{displaymath}
  which yields the desired convergence result.
  Furthermore, the desired error estimate holds for the initial data 
  $U_0 \in H^s$ with $s > 0$ by noting that 
  \begin{displaymath}
    \mathbbm{1}_{|\cdot| > \Xi}(\xi) \le 
    \frac{1 + |\xi|^s}{1 + \Xi^s},
  \end{displaymath}
  for each $\xi \in \bR^d$.
\end{proof}

\subsection{Asymptotic convergence and error estimate}
We are now ready to prove our main result on the asymptotic 
convergence of slow variables in the zero relaxation limit and the 
corresponding error estimate by applying the estimates in 
\Cref{lemma:low_frequency,lemma:high_frequency}.

\begin{proof}[Proof of \Cref{thm}]
  By \Cref{lemma:UQS_of_F1}, the stiff well-posedness of \cref{eq}
  ensures the strong hyperbolicity of reduced systems presented in 
  \cref{eq:reduced}.
  Given arbitrary $\mu > 0$ and $t > 0$, let us take a sufficiently 
  large $\Xi > 0$ such that 
  \begin{displaymath}
    \sup_{\eta \ge 0}
    \norm{ \big( \hat{u}(t, \cdot, \eta) - \exp(H_{11} t) \hat{u}_0
    \big) \mathbbm{1}_{|\cdot| > \Xi} }_{L^2} 
    < \frac{1}{2} \mu,
  \end{displaymath}
  according to \Cref{lemma:high_frequency}. 
  By \Cref{lemma:low_frequency}, there exists a constant 
  $\Theta > 0$ such that the following inequality holds for each 
  $\eta > \Theta$:
  \begin{displaymath}
    \norm{ \big( \hat{u}(t, \cdot, \eta) - \exp(H_{11} t) \hat{u}_0
    \big) \mathbbm{1}_{|\cdot| \le \Xi} }_{L^2} 
    < \frac{1}{2} \mu.
  \end{displaymath}
  At this time, one can obtain that 
  \begin{displaymath}
    \norm{ \big( \hat{u}(t, \cdot, \eta) - \exp(H_{11} t) \hat{u}_0
    \big) }_{L^2} < \mu,
  \end{displaymath}
  which yields the desired convergence.

  Furthermore, if the matrix set $\cF_2$ is uniformly quasi-stable 
  and the initial data satisfies $U_0 \in H^2$, then the 
  error estimate at the low-frequency region with the frequency 
  cutoff $\Xi(\eta) = \tilde{\beta}^{-1} \eta - 1$ can be given by 
  \cref{eq:low_frequency_error} in \Cref{lemma:low_frequency}.
  At the high-frequency region, the error can be bounded by 
  \begin{displaymath}
    \norm{ \big( \hat{u}(t, \cdot, \eta) - \exp(H_{11} t) \hat{u}_0
    \big) \mathbbm{1}_{|\cdot| > \Xi(\eta)} }_{L^2} \le 
    \frac{C}{1 + \Xi(\eta)^2} \|U_0\|_{H^2}.
  \end{displaymath}
  Combining the high- and low-frequency estimates, we complete 
  the proof of this theorem.
\end{proof}

\section{Conclusions}
\label{sec:concl}
In summary, this study advances the 
understanding of hyperbolic relaxation systems, focusing on the 
strong asymptotic convergence of slow variables under 
the minimal stability condition, leveraging the Mori--Zwanzig 
technique, frequency decomposition, and improved resolvent estimates. 
The Lax equivalence theorem for hyperbolic relaxation approximations 
is pivotal in this theory.

Our study reveals the strong asymptotic convergence of 
macroscopic variables in oscillatory systems, providing insights into 
fields like molecular dynamics simulations. 
The novel technique introduced in the proof, utilizing the Laplace 
transform and improved resolvent estimates, has the potential to 
enhance the estimation of solutions for well-posed Cauchy problems, 
especially in addressing resonance parts in oscillatory systems. 
Furthermore, the indirect approach to bound the high-frequency 
error, which involves investigating the reduced systems first, may 
inspire approaches to similar problems with small parameters.

\appendix
\section{Discussions of \Cref{thm:IKMT}}
\label{app:kreiss}
In this appendix, we will give a proof of \Cref{thm:IKMT} completely 
based on the Kreiss matrix theorem presented in \Cref{thm:KMT}.
We will also discuss several extensions of \Cref{thm:IKMT}, including 
resolvent estimates at any point in the resolvent set and the 
corresponding results for power-bounded matrices.

\subsection{Proof of \Cref{thm:IKMT}}

Using \Cref{lemma:QS_character}, we can prove that the quantity 
$\cK(M)$ is a measurement of quasi-stability of a 
given matrix $M$.
\begin{lemma}\label{lemma:QS_KM}
  Let $M \in \bC^{n \times n}$. The matrix $M$ is quasi-stable, if 
  and only if $\cK(M) < + \infty$.
\end{lemma}
\begin{proof}
  First, we assume that the matrix $M$ is quasi-stable. 
  \Cref{lemma:QS_character} yields that the spectrum $\sigma(M)$ is 
  a subset of the closed left half-plane $\bH^\rc$.
  Let $J = S M S^{-1}$ be the Jordan matrix of $M$. We notice that 
  \begin{displaymath}
    \norm{(z I - M)^{-1}} \le \norm{S^{-1}} \cdot 
    \norm{(z I - J)^{-1}} \cdot \norm{S}.
  \end{displaymath}
  To estimate $\norm{(z I - J)^{-1}}$, it suffices to estimate the 
  resolvent of each Jordan block. Let 
  \begin{displaymath}
    J_r = \begin{pmatrix}
      \lambda_r & 1 & & \\
      & \ddots & \ddots & \\
      & & \ddots & 1 \\
      & & & \lambda_r
    \end{pmatrix} = \lambda_r I + N \in \bC^{k \times k}
  \end{displaymath}
  be an arbitrary Jordan block of $J$.
  By direct calculation, one has that 
  \begin{equation}\label{equ:resolventJr}
    (z I - J_r)^{-1} = \sum_{j = 0}^{k - 1} 
      (z - \lambda_r)^{-(j + 1)} N^j, \quad z \in \bH.
  \end{equation}
  If $\Re(\lambda_r) < 0$, then for each $z \in \bH$, one has that 
  \begin{displaymath}
    \begin{aligned}
      \norm{(z I - J_r)^{-1}} & \le \left( 
        \sum_{j = 0}^{k - 1} |z - \lambda_r|^{-j} \norm{N}^j \right)
        \cdot |z - \lambda_r|^{-1} \\
      & \le \left( \sum_{j = 0}^{k - 1} |\Re(\lambda_r)|^{-j} 
        \norm{N}^j \right) \cdot 
        \max_{\lambda \in \sigma(M) \setminus \bH} |z - \lambda|^{-1}.
    \end{aligned}
  \end{displaymath}
  If $\Re(\lambda_r) = 0$, then \Cref{lemma:QS_character} yields that 
  the order of $J_r$ is $k = 1$, and thus 
  \begin{displaymath}
    \norm{(z I - J_r)^{-1}} = |z - \lambda_r|^{-1} \le 
    \max_{\lambda \in \sigma(M) \setminus \bH} |z - \lambda|^{-1},
    \quad z \in \bH.
  \end{displaymath}
  Therefore, we obtain that $\cK(M) < + \infty$.

  Conversely, we assume that $\cK(M) < + \infty$. Let $\lambda_r$ be 
  an arbitrary eigenvalue of $M$, and $J_r \in \bC^{k \times k}$ be 
  the largest Jordan block corresponding to $\lambda_r$. By 
  \cref{equ:resolventJr}, one can obtain that 
  \begin{displaymath}
    \norm{(z I - M)^{-1}} \sim |z - \lambda_r|^{-k}, \quad 
    \text{ as } z \rightarrow \lambda_r.
  \end{displaymath}
  If $\Re(\lambda_r) > 0$, then 
  \begin{displaymath}
    \frac{\norm{(z I - M)^{-1}}}
    {\max_{\lambda \in \sigma(M) \setminus \bH} |z - \lambda|^{-1}}
    \sim |z - \lambda_r|^{-k}, \quad \text{ as } 
    z \rightarrow \lambda_r.
  \end{displaymath}
  If $\Re(\lambda_r) = 0$ and $k > 1$, then 
  \begin{displaymath}
    \frac{\norm{(z I - M)^{-1}}}
    {\max_{\lambda \in \sigma(M) \setminus \bH} |z - \lambda|^{-1}}
    \sim |z - \lambda_r|^{- k + 1}, \quad \text{ as }
    z \rightarrow \lambda_r.
  \end{displaymath}
  Both cases yield a contradiction, and then the quasi-stability of 
  the matrix $M$ follows from \Cref{lemma:QS_character}.
\end{proof}

Here is an elementary fact from linear algebra. It asserts that 
for a set of unit upper triangular matrices that are upper 
triangular matrices with $1$ on the diagonal.
in $\bC^{n \times n}$, uniform boundedness is equivalent to uniform 
boundedness of inverse.
\begin{lemma}\label{lemma:unittri}
  If the matrix $A \in \bC^{n \times n}$ is a unit upper triangular 
  matrix with $\norm{A} \le \alpha$, then one has that 
  $\norm{A^{-1}} \le (n \alpha)^{n - 1}$.
\end{lemma}
\begin{proof}
  Let the column partitioning of the strictly upper triangular 
  component of $A$ be 
  \begin{displaymath}
    U := A - I = \begin{pmatrix}
      0 & u_2 & \cdots & u_n
    \end{pmatrix}.
  \end{displaymath}
  Denote the $i$th column of the identity matrix $I$ as $e_i$.
  It is clear that $e_i^\top u_j = 0$ for $2 \le j \le i \le n$.
  Therefore, one can obtain that 
  \begin{multline*}
    A = I + u_n e_n^\top + u_{n - 1} e_{n - 1}^\top + 
    \ldots + u_2 e_2^\top \\ = \left( I + u_n e_n^\top \right) 
    \left( I + u_{n - 1} e_{n - 1}^\top \right)
    \ldots \left( I + u_2 e_2^\top \right),
  \end{multline*}
  and that 
  \begin{displaymath}
    I = \left( I + u_j e_j^\top \right) 
    \left( I - u_j e_j^\top \right), \quad 
    \text{ for } j = 2, 3, \ldots, n,
  \end{displaymath}
  which yields that 
  \begin{displaymath}
    A^{-1} = \left( I - u_2 e_2^\top \right)
    \left( I - u_3 e_3^\top \right) \ldots 
    \left( I - u_n e_n^\top \right).
  \end{displaymath}
  By the equivalence of $1$-norm and $2$-norm for matrices, one can 
  obtain that
  \begin{displaymath}
    \norm{I - u_j e_j^\top} \le \sqrt{n} \norm{I - u_j e_j^\top}_1 
    \le \sqrt{n} \norm{A}_1 \le n \norm{A} \le n \alpha, 
  \end{displaymath}
  for each $j = 2, 3, \ldots, n$.
  Therefore, one can obtain that 
  \begin{displaymath}
    \norm{A^{-1}} \le \norm{I - u_2 e_2^\top}
    \norm{I - u_3 e_3^\top} \ldots 
    \norm{I - u_n e_n^\top} \le (n \alpha)^{n - 1},
  \end{displaymath}
  which completes the proof.
\end{proof}

We are now set to prove \Cref{thm:IKMT}.
\begin{proof}[Proof of \cref{thm:IKMT}]
  First, we assume that the set of matrices $\cF$ satisfies 
  \cref{eq:uniformKM}. By \Cref{lemma:QS_KM}, each matrix in $\cF$ 
  is quasi-stable. For each $M \in \cF$, $z \in \bH$, and 
  $\lambda \in \sigma(M) = \sigma(M) \setminus \bH$, one has that 
  \begin{displaymath}
    0 < \Re(z) \le \Re(z - \lambda) \le |z - \lambda|,
  \end{displaymath}
  which yields that 
  \begin{multline*}
    \sup_{M \in \cF} \sup_{z \in \bH} \left( \Re(z) \cdot 
    \norm{(z I - M)^{-1}} \right) \\ \le \sup_{M \in \cF} 
    \sup_{z \in \bH} \frac{\norm{(z I - M)^{-1}}}
    {\max_{\lambda \in \sigma(M) \setminus \bH} |z - \lambda|^{-1}}
    = \sup_{M \in \cF} \cK(M) < + \infty.
  \end{multline*}
  Therefore, the set of matrices $\cF$ is uniformly quasi-stable 
  thanks to \Cref{thm:KMT}.

  Conversely, we assume that $\cF$ is uniformly quasi-stable. At this 
  time, the spectrum $\sigma(M)$ is a subset of the closed left 
  half-plane $\bH^\rc$ for each $M \in \cF$. 
  Using the notations in the third 
  condition of \Cref{thm:KMT}, we denote the diagonal component and 
  the strictly upper triangular component of $S M S^{-1}$ by $D$ and 
  $N$, respectively, that is, 
  \begin{displaymath}
    D = D(M) = \diag \left\{ b_{11}, b_{22}, \ldots, b_{nn} \right\},
    \quad N = N(M) = S M S^{-1} - D.
  \end{displaymath}
  By direct calculation, for any $z \in \bH$, one has that 
  \begin{displaymath}
    (z I - M)^{-1} = S^{-1} (z I - D - N)^{-1} S = 
    S^{-1} \left( I - (z I - D)^{-1} N \right)^{-1} (z I - D)^{-1} S,
  \end{displaymath}
  which yields that 
  \begin{displaymath}
    \norm{(z I - M)^{-1}} \le \kappa(S) \cdot 
    \norm{\left( I - (z I - D)^{-1} N \right)^{-1}} \cdot 
    \norm{(z I - D)^{-1}},
  \end{displaymath}
  where the condition number of $S$ can be estimated as 
  \begin{displaymath}
    \kappa(S) := \norm{S} \cdot \norm{S^{-1}} \le 
    \frac{\left( \norm{S} + \norm{S^{-1}} \right)^2}{4} \le 
    \frac{K_{31}^2}{4}.
  \end{displaymath}
  For each $1 \le i < j \le n$, the absolute value of the $(i,j)$-th 
  entry of $(z I - D)^{-1} N$ is 
  \begin{equation}\label{equ:entry_estimate}
    \frac{|b_{ij}|}{|z - b_{ii}|} \le 
    \frac{K_{32} |\Re(b_{ii})|}{|\Re(z) - \Re(b_{ii})|} \le K_{32}.
  \end{equation}
  \Cref{lemma:unittri} yields that there exists a constant $C > 0$ 
  dependent on $n$ and $K_{32}$ such that 
  \begin{displaymath}
    \norm{\left( I - (z I - D)^{-1} N \right)^{-1}} \le C,
  \end{displaymath}
  since the matrix $I - (z I - D)^{-1} N$ is unit upper triangular.
  We notice that 
  \begin{displaymath}
    \norm{(z I - D)^{-1}} = \max_{\lambda \in \sigma(M)} 
    |z - \lambda|^{-1} = 
    \max_{\lambda \in \sigma(M) \setminus \bH} 
    |z - \lambda|^{-1}.
  \end{displaymath}
  Therefore, we have that 
  \begin{displaymath}
    \sup_{M \in \cF} \cK(M) \le \frac{K_{31}^2 C}{4} < + \infty,
  \end{displaymath}
  which completes the proof of the theorem.
\end{proof}

\subsection{Extensions}
Let us generalize \Cref{thm:IKMT} and discuss its 
relations with several previous results.

\subsubsection{Resolvent estimate in the resolvent set}
From the proof of the necessity part of \Cref{thm:IKMT}, one can 
observe that it is not essential to take $z$ in the open right 
half-plane $\bH$ to obtain a resolvent estimate as in 
\cref{eq:main_resolvent}. Given a matrix $M \in \cF$ and a number 
$r > 0$, let us denote 
\begin{displaymath}
  \cS(M, r) := 
  \left\{ z \in \bC : \max_{\lambda \in \sigma(M)} 
  \frac{|\Re(\lambda)|}{|z - \lambda|} \le \frac{1}{r} \right\}.
\end{displaymath}
The following theorem provides a resolvent estimate at any point in 
the resolvent set. Similar results can be found in 
\cite{miller_resolvent_1968}.
\begin{theorem}\label{thm:main_miller}
  If a set of matrices $\cF$ is uniformly quasi-stable, then there 
  exists a positive constant $K$ such that 
  \begin{displaymath}
    \norm{(z I - M)^{-1}} \le K \left( 1 + \frac{1}{r} \right)^{n - 1}
    \max_{\lambda \in \sigma(M)} |z - \lambda|^{-1},
  \end{displaymath}
  for all $M \in \cF$, $r > 0$ and $z \in \cS(M, r)$.
\end{theorem}
\begin{proof}
  Let us take $M \in \cF$, $r > 0$ and $z \in \cS(M, r)$. Using the 
  notations in the third condition of \Cref{thm:KMT}, for each 
  $1 \le i < j \le n$, one can obtain that 
  \begin{displaymath}
    \frac{|b_{ij}|}{|z - b_{ii}|} \le 
    \frac{K_{32} |\Re(b_{ii})|}{|z - b_{ii}|} \le \frac{K_{32}}{r},
  \end{displaymath}
  which is analogous to \cref{equ:entry_estimate}. 
  According to \Cref{lemma:unittri}, it follows that there exists a 
  constant $C$ dependent on $n$ and $K_{32}$ such that 
  \begin{displaymath}
    \norm{\left( I - (z I - D)^{-1} N \right)^{-1}} \le 
    C \left( 1 + \frac{1}{r} \right)^{n - 1}.
  \end{displaymath}
  Hence the proof is completed by a similar argument in the proof of 
  \Cref{thm:IKMT}.
\end{proof}

\subsubsection{Power-bounded matrices}
Analogous results for power-bounded matrices can be proven similarly. 
Let us recall the following Kreiss matrix theorem 
\cite{kreiss_stabilitatsdefinition_1962} for power-bounded matrices.
\begin{theorem}[Kreiss matrix theorem. See Satz 4.1 of 
  \cite{kreiss_stabilitatsdefinition_1962}]
  \label{thm:KMT_power}
Let $\cF$ denote a set of matrices in $\bC^{n \times n}$.
The following four conditions are equivalent.
\begin{enumerate}
  \item There exists a constant $K_1$ such that  
    $\norm{M^\nu} \le K_1$ for all $M \in \cF$ and 
    $\nu \in \bN$.
  \item There exists a constant $K_2$ such that 
    \begin{displaymath}
      \norm{(z I - M)^{-1}} \le \frac{K_2}{|z| - 1}, 
    \end{displaymath}
    for all $M \in \cF$ and $z \in \bC$ with $|z| > 1$.
  \item There exist constants $K_{31}$, $K_{32}$ such that 
    for each $M \in \cF$, there exists a 
    transformation $S = S(M)$ with 
    $\norm{S} + \norm{S^{-1}} \le K_{31}$,
    the matrix $S M S^{-1}$ is upper triangular,
    \begin{displaymath}
      S M S^{-1} = \begin{pmatrix}
        b_{11} & b_{12} & \cdots & b_{1n} \\
              & b_{22} & \cdots & b_{2n} \\
              &    & \ddots & \vdots \\
              &    &    & b_{nn} \\
      \end{pmatrix},
    \end{displaymath}
    the diagonal is ordered, 
    \begin{displaymath}
      1 \ge |b_{11}| \ge |b_{22}| \ge \cdots \ge 
      |b_{nn}|,
    \end{displaymath}
    and the upper diagonal elements satisfy the estimate 
    \begin{displaymath}
      |b_{ij}| \le K_{32} (1 - |b_{ii}|), \quad 
        1 \le i < j \le n.
    \end{displaymath}
  \item There exists a positive constant $K_4$ such that 
    for each $M \in \cF$, there exists a Hermitian 
    matrix $H = H(M)$ such that 
    \begin{displaymath}
      K_4^{-1} I \le H \le K_4 I, \quad M^* H M \le H.
    \end{displaymath}
\end{enumerate}
\end{theorem}

Similar to \Cref{thm:main_miller}, we would like to give a 
resolvent estimate for a set of uniformly power-bounded matrices 
$\cF$ at any point in the resolvent set. Given a matrix $M \in \cF$ 
and a number $r > 0$, let us denote 
\begin{displaymath}
  \cT(M, r) := 
  \left\{ z \in \bC : \max_{\lambda \in \sigma(M)} 
  \frac{1 - |\lambda|}{|z - \lambda|} \le \frac{1}{r} \right\}.
\end{displaymath}
We have the following theorem. The second conclusion of 
\Cref{thm:main_power} can be found in \cite{zarouf_sharpening_2009}.
\begin{theorem}\label{thm:main_power}
  If a set of matrices $\cF$ satisfies any condition of 
  \Cref{thm:KMT_power}, then there exists a positive constant $K$ 
  such that 
  \begin{displaymath}
    \norm{(z I - M)^{-1}} \le K \left( 1 + \frac{1}{r} \right)^{n - 1}
    \max_{\lambda \in \sigma(M)} |z - \lambda|^{-1},
  \end{displaymath}
  for all $M \in \cF$, $r > 0$ and $z \in \cT(M, r)$. In particular, 
  one has that 
  \begin{displaymath}
    \sup_{M \in \cF} \sup_{|z| > 1} 
    \frac{\norm{(z I - M)^{-1}}}{\max_{\lambda \in \sigma(M)} 
    |z - \lambda|^{-1}} < + \infty.
  \end{displaymath}
\end{theorem}

\begin{proof}
  Let us take $M \in \cF$, $r > 0$ and $z \in \cT(M, r)$. 
  Using the notations in the third condition of \Cref{thm:KMT_power}, 
  for each $1 \le i < j \le n$, one can obtain that 
  \begin{displaymath}
    \frac{|b_{ij}|}{|z - b_{ii}|} \le 
    \frac{K_{32} (1 - |b_{ii}|)}{|z - b_{ii}|} \le \frac{K_{32}}{r}.
  \end{displaymath}
  The remainder of the proof of the first conclusion is analogous to 
  that in \Cref{thm:main_miller}. The second conclusion follows from 
  an observation that 
  \begin{displaymath}
    |z - \lambda| \ge |z| - |\lambda| > 1 - |\lambda| \ge 0,
  \end{displaymath}
  for each $\lambda \in \sigma(M)$ and $z \in \bC$ with $|z| > 1$, 
  and thus $z \in \cT(M, 1)$.
\end{proof}

\end{document}